\documentclass[reqno,11pt]{amsart}
\usepackage{amsmath,amsfonts,amssymb,amsthm}\usepackage[foot]{amsaddr}
\usepackage{mathtools,mathrsfs,yhmath}
\usepackage{graphicx,color,xcolor}
\usepackage{enumitem}
\usepackage{hyperref}
\usepackage[svgnames,x11names,table]{xcolor}
\hypersetup{colorlinks=true, linkcolor=blue, filecolor=magenta, urlcolor=DarkBlue}
\usepackage[capitalize,nameinlink]{cleveref}
 \allowdisplaybreaks
\numberwithin{equation}{section}

\newcommand{\R}{\mathbb{R}}

\newcommand{\HW}{\mathrm{HW}}

\newcommand{\X}{\mathbb{R}^d}

\newcommand{\Rd}{\R^d}

\newcommand{\XRd}{\R^{2d}}
\DeclarePairedDelimiter{\abs}{\lvert}{\rvert}
\DeclarePairedDelimiter{\norm}{\lVert}{\rVert}

\mathtoolsset{showonlyrefs}

\numberwithin{equation}{section}
\newtheorem{theorem}{Theorem}[section]
\newtheorem{lemma}[theorem]{Lemma}
\newtheorem{proposition}[theorem]{Proposition}
\newtheorem{corollary}[theorem]{Corollary}

\newtheorem{remark}[theorem]{Remark}

\makeatletter
\let\oldabs\abs
\def\abs{\@ifstar{\oldabs}{\oldabs*}}
\let\oldnorm\norm
\def\norm{\@ifstar{\oldnorm}{\oldnorm*}}
\makeatother

\title[Moment Propagation for Magnetized Vlasov-Poisson]{Propagation of Velocity Moments for the Magnetized Vlasov-Poisson System with Space-Time Dependent Magnetic Fields}
\author{Immanuel Ben-Porat}
\address{University of Basel, Department of Mathematics and Computer Science, Spiegelgasse 1, 4051 Basel, Switzerland}
\email{immanuel.ben-porath@unibas.ch}

\author{Antoine Gagnebin}
\address{ETH Zürich, Department of Mathematics, Rämistrasse 101, 8092 Zürich, Switzerland}
\email{antoine.gagnebin@math.ethz.ch}

\author{Mikaela Iacobelli}
\address{ETH Zürich, Department of Mathematics, Rämistrasse 101, 8092 Zürich, Switzerland}
\email{mikaela.iacobelli@math.ethz.ch}

\author{Jonathan Junné}
\address{TU Delft, Delft Institute of Applied Mathematics, Mekelweg 4, 2628 CD Delft, Netherlands}
\email{j.junne@tudelft.nl}

\date{\today}

\begin{document}

\begin{abstract}
We prove that polynomial velocity moments of solutions to the 2D magnetized Vlasov-Poisson system and the 3D magnetized screened Vlasov-Poisson equation remain finite for all times, provided they are finite initially, even when the external magnetic field $B=B(t,x)$ is space-time dependent. We deduce propagation of regularity, thereby implying the existence of global classical solutions. Moreover, we prove optimal stability estimates in the kinetic-Wasserstein distance on par with the unmagnetised case.

\end{abstract}
\maketitle

\section{Introduction}
In this work, we study propagation of moments  for the \textit{magnetized} Vlasov-Poisson equation. The  Cauchy problem of the magnetized Vlasov-Poisson equation reads 
\begin{equation}
\left\{ \begin{array}{lc}
\partial_{t}f+v\cdot \nabla_{x}f-(\nabla_{x}K\star \rho_{f}+v\wedge B)\cdot\nabla_{v}f=0,\ \rho_{f}(t,x)\coloneqq\int_{\mathbb{R}^{d}}f(t,x, v)\ d v\\
f|_{t=0}=f^{\mathrm{in}}.
\end{array}\right.\label{MAGNETIC Vlasov Intro}
\end{equation}
The unknown is a time-dependent, non-negative bounded integrable function $f(t,\cdot,\cdot)\in L^{1}\cap L^{\infty}(\mathbb{R}^{d}_{x}\times \mathbb{R}^{d}_{ v})$ and $B=B(t,x)$ is a time-dependent vector field    
$B:[0,T]\times \mathbb{R}^{d}\rightarrow \mathbb{R}^{d}$ called the \textit{magnetic field}.  The potential $K$ is of Coulomb type in 2D and of Yukawa/screened-Coulomb type in 3D. More specifically, $K$ is given by
\begin{align}
K(x)\coloneqq \begin{cases}
\begin{array}{ll}
\pm \frac{1}{2\pi}\log(x) & d=2,\\
\pm \frac{e^{-\kappa\left\vert x\right\vert }}{4\pi\left\vert x\right\vert} & d=3, \ \kappa>0,
\end{array}
\end{cases}
\label{Coulomb def}
\end{align}
and 
\begin{align*}
    \nabla_{x}K\star \rho_{f} (t,x):= \int_{\mathbb{R}^{d}}\nabla_{x}K (x-y) \rho_{f}(t,y)\ dy.
\end{align*}
Note that in the case where $\kappa=0$ we recover the usual Coulomb potential. In the absence of a magnetic field, i.e., when $B=0$, equation \eqref{MAGNETIC Vlasov Intro} reduces to the classical Vlasov-Poisson equation  
\begin{equation}
\left\{ \begin{array}{lc}
\partial_{t}f+v\cdot\nabla_{x}f-\nabla_{x}K\star \rho_{f}\cdot\nabla_{v}f =0,\ \rho_{f}(t,x)\coloneqq\int_{\mathbb{R}^{d}}f(t,x, v)\ d v\\
f|_{t=0}=f^{\mathrm{in}}.
\end{array}\right.\label{nonMAGNETIC Vlasov Intro}
\end{equation}

Such systems arise naturally in plasma physics and galactic dynamics (see, e.g. \cite{plasma_2003, Galactic_2008, Ryutov_1999, Villani_notes}). In this work we mainly focus on plasmas, where $f(t,x,v)$ denotes the distribution function of particles at time $t$, position $x$, and velocity $v$.  A plasma is an ionized gas formed when a neutral gas is exposed to high temperatures or strong electromagnetic fields. The ionization process produces two species: light, negatively charged electrons and heavy, positively charged ions. These charged particles interact through long-range electromagnetic forces, giving rise to collective effects in the plasma. Because of the large mass difference between electrons and ions, the two species evolve on distinct time scales.

Electrons move much faster than ions. When studying the electron dynamics, it is thus common to assume that the ions are stationary. The Vlasov-Poisson equation ($\kappa = 0$) is a fundamental model for the dynamics of electrons in a plasma when collisions are neglected and ions are treated as a fixed background.

From the ions’ point of view, the electrons move much faster, and electron-electron collisions become relevant on the ion timescale. It is therefore natural to assume that the electron distribution has already relaxed to thermal equilibrium, so that the electron density follows a Maxwell-Boltzmann law. By linearizing this equilibrium, we obtain the screened Vlasov-Poisson equation ($\kappa > 0$), which provides an effective description of ion dynamics.

In the presence of a magnetic field $B$, the magnetized Vlasov-Poisson system \eqref{MAGNETIC Vlasov Intro} describes the evolution of a collection of charged particles interacting through the Coulomb or screened-Coulomb potential and subject to the additional influence of the external magnetic field $B$.

The existence and uniqueness theory of \eqref{nonMAGNETIC Vlasov Intro} is well understood, both for the cases where $\kappa=0$ (corresponding to Coulomb) and $\kappa>0$ (corresponding to screened Coulomb): the Vlasov-Poisson system has been the subject of extensive study over the past decades. Many results are now available concerning the existence of global classical and weak solutions under different assumptions on the initial data. The first local-in-time existence result in three dimensions was obtained by Rudolf \cite{Rudolf}. For weak solutions on the whole space, we refer to the work of Arsenev \cite{arsenev1975existence} (see also \cite{Bardos_Degond_1, Bardos_Degond_Golse, Horst_Hunze}). Global well-posedness of classical solutions on the whole space was later established in one dimension by Iordanskii \cite{iordanskii1961cauchy}, in two dimensions by Ukai and Okabe \cite{ukai1978classical}, and in three dimensions by   Pfaffelmoser \cite{pfaffelmoser1992global} (see also \cite{bardos1985global, Schaffer}). The seminal work of Lions and Perthame \cite{lions1991propagation} further studies the propagation of velocity moments, from which global existence of smooth classical solutions can also be deduced.
For the periodic case, i.e., where the space domain is the flat torus and $K$ is given by means of the periodic Green kernel, global well-posedness was proved by Batt and Rein \cite{batt1991global}, with further developments by Pallard in \cite{pallard2012moment} and by Chen and Chen in \cite{ChenChen2019}. 
The weak global existence of the screened Vlasov-Poisson equation is discussed by Han-Kwan in \cite[Theorem 2.1]{HKD}. The proof of this result is an adaptation of Arsenev \cite{arsenev1975existence}. For a more detailed discussion about this system, the interested reader can look into \cite[Section 1.1 and 1.2]{HKD_HDR}. We also refer to \cite{golse2016dynamics} and \cite{griffin2020recent} for a more detailed account on the state of the art of well-posedness results concerning Vlasov type equations.


Questions of stability and uniqueness are also vast: the epitome of all kinetic stability estimates is the work of Dobrushin \cite{dobrushin1979vlasov}, which proves stability for measure valued solutions and thus is closely linked to the problem of deriving the Vlasov equation as a mean field limit -- see \cite{jabin2014review} for a detailed account on the latter theme. However, Dobrushin's approach is limited to interactions $K$ which are at least $C^{1,1}$, and is thus far from being applicable to Coulombic singularities. For the full space Coulomb case, Loeper \cite{loeper2006uniqueness} proved uniqueness of solutions of \eqref{nonMAGNETIC Vlasov Intro} via a stability estimate with respect to the 2-Wasserstein distance. Later, Holding and Miot extended Loeper’s uniqueness criterion for the Vlasov–Poisson system to solutions whose associated density belongs to suitable Orlicz spaces \cite{Miot_2016, Holding_Miot} (see also \cite{Crippa_co}); while Han-Kwan and Iacobelli adapted Loeper's argument to obtain a stability estimate on the torus in the context of the study of quasineutral limits \cite{han2017quasineutral}. 

The inclusion of a magnetic field introduces new difficulties for what concerns the propagation of velocity moments and  the  stability estimates. The case of a uniform magnetic field, constant in both time and space, was first studied by Rege \cite{rege2021uniformB}, who established propagation of velocity moments for the magnetized Vlasov--Poisson equation in the 3D whole space case. In a subsequent work \cite{rege2023propagation}, the same author considered the case of a spatially uniform but time-dependent magnetic field, $B(t)$.
In \cite{rege2025stability}, he proved stability estimates \`a la Loeper for the magnetized Vlasov--Poisson equation on the torus with non-uniform magnetic field. However, the more general question of proving global propagation of velocity moments for the Vlasov--Poisson equation with non-uniform and non-constant magnetic field, remained open.

Beyond moment propagation, other directions of research have focused on spectral and numerical aspects of the magnetized Vlasov--Poisson equation. In particular, Charles, Després, Rege, and Weder analysed the Bernstein-Landau paradox and related spectral phenomena in \cite{Charles_co_Landau}, while several efficient numerical schemes for the magnetized Vlasov--Poisson system have been proposed in \cite{Charles_co_numeric, Fibet_Rodrigues_2016, Filbet_co_2021, Fibet_Rodrigues_2023, Filbet_co_2025}.

Further related developments address asymptotic and confining effects induced by strong external magnetic fields. Caprino, Cavallaro, and Marchioro studied the dynamics of plasmas confined in an infinite cylinder by an unbounded magnetic field in \cite{Caprino_2012}, while Knopf and Weber established global well-posedness and nonlinear stability of confined steady states for the two-and-one-half dimensional Vlasov--Poisson system in \cite{Knopf_Weber_2022}. The asymptotic behaviour of the Vlasov--Poisson system under strong magnetic fields has also been extensively investigated by Golse and Saint-Raymond, notably in the quasineutral regime and in the derivation of guiding-center or gyrokinetic models \cite{GSR98,GSR99,GSR2003}. Frénod and Sonnendrücker analysed the homogenization of the Vlasov--Poisson system in such regimes in \cite{Frenod_Sonnendrucker_98}. They also derived in \cite{Frenod_Sonnendrucker_2001} the finite Larmor radius approximation, describing the asymptotic behaviour of charged particles under strong external magnetic fields (see also the works of Han-Kwan on that topic \cite{Han_Kwan_2010, Han_Kwan_2012, Han_Kwan_2013}). In the same direction, effective models governing the finite Larmor radius regime and the quasineutral limit for strongly magnetized plasmas were obtained by Bostan, Finot, and Hauray in \cite{Bostan_Finot, Bostan_2016, Bostan_2020}. In \cite{Degond_Filbet_2016}, Degond and Filbet established the long-time asymptotic limit of the three-dimensional Vlasov--Poisson equation in the strong-field regime, including the guiding-center approximation in the case of non-uniform magnetic fields. We also mention the result of Herda \cite{Herda_2016}, who studied related asymptotic regimes for multispecies systems in the massless-electron limit and in the presence of external magnetic fields.

In the present work, we investigate the propagation of moments, propagation of regularity and Wasserstein stability in the magnetized setting of equation \eqref{MAGNETIC Vlasov Intro}.

Given a solution $f$ of \eqref{MAGNETIC Vlasov Intro} or \eqref{nonMAGNETIC Vlasov Intro} we define the $n$-th velocity moments as
\begin{align}
M_{n}(t)\coloneqq \int_{\mathbb{R}^{2d}}\left\vert  v \right\vert^{n}f(t,x, v) \ dxd v  \label{non magnetic velocity moments intro}  \end{align}
and the $n$-th position moments as 
\begin{align}
N_{n}(t)\coloneqq \int_{\mathbb{R}^{2d}} \left\vert x\right\vert^{n}f(t,x, v)\ dxdv. \label{position moment} 
\end{align}
When we say that $M_{n}(t)$ \textit{is propagated in time} we mean  that there exists a function $\Phi_{n}(t)=\Phi_{n}\left(\left\Vert f^{\mathrm{in}}\right\Vert_{\infty},M_{n}(0) ,t\right)$ continuous in $t$ such that
\begin{align}
M_{n}(t)\leq \Phi_{n}(t) \text{ for all } t\in[0,T]. \label{def of propagation}    
\end{align}
The propagation of velocity moments is a central ingredient in proving existence of classical solutions, and goes back to the work of Lions and Perthame \cite{lions1991propagation}. Indeed, once we have uniform bounds on the velocity moments $M_{n}$ for all times, we can show that the spatial density $\rho_f$ also remains bounded. Controlling the density then provides bounds on the force field, $\nabla_x K \star \rho_f$, ensuring that it remains sufficiently regular. When both $f$ and $\nabla_{x} K\star\rho_{f}$ are controlled in suitable norms, the characteristics of the Vlasov equation stay regular for all time, and consequently leads to the global existence of classical solutions. Extensions and improvements of the work of Lions and Perthame \cite{lions1991propagation} have been investigated by Pallard in \cite{pallard2012moment}. A self contained exposition to propagation of velocity moments for potentials with even stronger singularities than Coulomb (local or global) can be found in Lafleche \cite{lafleche2019propagation,lafleche2021global}, which also studies the quantum analogue of this problem and the semi-classical limit. 

As already mentioned, the aim of the current work is to address propagation of moments for the 2D and 3D full space case of the system \eqref{MAGNETIC Vlasov Intro}. Note that in the problem we are considering the magnetic field is external. Propagation of moments for the Vlasov--Maxwell equation, which corresponds to the case where the magnetic field is self-consistent, remains an outstanding open problem in kinetic theory. Nonetheless, several important contributions have significantly advanced our understanding of the Vlasov--Maxwell equation. These include the classical works of Glassey and Strauss in \cite{Glassey_Strauss}, Klainerman and Staffilani in \cite{KS02}, and Bouchut, Golse and Pallard in \cite{Bouchut_Golse_Pallard_2003}, as well as more recent developments \cite{LS14, LS16, Preissl_co_2021}.

In order to state our main result, we need to introduce hypotheses on the external magnetic field $B(t,x)$. Let
\begin{align*}
\mathbf{b}(t,x)\coloneqq x\wedge B(t,x) . 
\end{align*}
Given $c\in (\frac{7}{5},\frac{3}{2})$ let $a=\frac{3}{c}-1$. Suppose that there is some constant $B_{0}>0$ such that $\mathbf{b}(t,x)$ and $B(t,x)$ satisfy:
\begin{equation}\tag{\textbf{B}}\label{eq:B}
\begin{aligned}
&\|\mathbf{b}(t,\cdot)\|_{\infty}\le B_{0}t^{-a}, \quad \text{and} \quad \|B(t,\cdot)\|_{\infty}\le B_{0}t^{-a}. 
\end{aligned}
\end{equation}
Clearly any smooth compactly supported function will satisfy \eqref{eq:B}. We can now state our first main theorem:

\begin{theorem}
Assume that $K$ is the kernel as in \eqref{Coulomb def}. Let $0\leq f^{\mathrm{in}}\in L^{1}(\mathbb{R}^{d}\times\mathbb{R}^{d})\cap L^{\infty}(\mathbb{R}^{d}\times\mathbb{R}^{d})$ and suppose there exist $\mathbf{M}^{\mathrm{in}}_{k}>0$ for $k=0,\dots,n$ such that
\[
M_{k}(0)\le \mathbf{M}^{\mathrm{in}}_{k}\qquad\text{for all }0\le k\le n.
\]
Let $f$ be a classical solution to \eqref{MAGNETIC Vlasov Intro} on $[0,T]\times\mathbb{R}^{d}$ with initial data $f^{\mathrm{in}}$.
Then,
\begin{itemize}
    \item [i)]  if $d=2$, there is some $\Phi_{n}(t)=\Phi_{n}\left(\left\Vert f^{\mathrm{in}}\right\Vert_{L^\infty\cap L^{1}},\mathbf{M}^{\mathrm{in}}_{n},t\right)$,  continuous in $t$, such that it holds that 
\begin{align*}
M_{n}(t)\leq \Phi_{n}(t) \ \mbox{for all}\ t\in [0,T].    
\end{align*}
\item [ii)] if $d=3$, assume in addition that $n\ge 4$, that hypothesis \eqref{eq:B} holds, and that there exist constants $\mathbf{N}^{\mathrm{in}}_{k}>0$ with
\[
N_{k}(0)\le \mathbf{N}^{\mathrm{in}}_{k}\qquad\text{for all }0\le k\le n.
\]
Then, there exists some constant $\varepsilon_{0}=\varepsilon_{0}(B_{0})>0$ such that if 
\begin{align}
\left\Vert f^{\mathrm{in}}\right\Vert_{L
^{\infty}\cap L^{1}}+M_{4}(0)+N_{4}(0)\leq \varepsilon_{0} \label{smallness assumption}   
\end{align}
then there is some $\Phi_{n}(t)=\Phi_{n}\left(\left\Vert f^{\mathrm{in}}\right\Vert_{L^{\infty}\cap L^{1}},\mathbf{M}^{\mathrm{in}}_{n},\mathbf{N}^{\mathrm{in}}_{n},t\right)$, continuous in $t$, such that 
\begin{align*}
M_{n}(t)\le \Phi_{n}(t)\qquad\text{and}\qquad N_{n}(t)\le \Phi_{n}(t)\qquad\text{for all }t\in[0,T].   
\end{align*}
\end{itemize}
\label{main thm intro}
\end{theorem}
One aspect in which our approach differs from the work of Rege \cite{rege2023propagation} lies in the fact that we do not appeal to the trajectories in the proof, but instead rely only on functional inequalities, namely on kinetic interpolation inequalities (see Lemma \ref{Magnetic-kinetic-interpolation}). To better understand why this approach is helpful, we recall that a crucial ingredient in the proof of Lions-Perthame is the following representation formula for the density: 
\begin{align}
\rho_{f}(t,x)=\mathrm{div}_{x}\int_{0}^{t}s\int_{\mathbb{R}^{d}}\nabla_{x} K\star \rho_{f}(t-s,x-sv)f(t,x-sv,v)\  dvds+\int_{\mathbb{R}^{d}}f^{\mathrm{in}}(x-tv,v)\ dv. \label{rep formula}   \end{align}
This formula is obtained by solving explicitly the system of characteristics corresponding to the potential- free Vlasov equation, i.e., solving the system 
\begin{align}
\begin{cases}
\begin{array}{cc}
\dot X (t,x,v)=V(t,x,v), & X(0,x,v)=x,\\
\dot V(t,x,v)=0, & V(0,x,v)=v. \label{trajectories of Vlasov free}
\end{array}\end{cases}    
\end{align}
The solution of \eqref{trajectories of Vlasov free} is given by $V(t,x,v)=v$ and $X(t,x,v)=x+tv$. Combined with Duhamel's formula, it is then straightforward to obtain \eqref{rep formula}. In the case where the magnetic field is uniform, the corresponding equation for the characteristics is still explicitly solvable, and the argument used by Rege in \cite{rege2021uniformB} relies on this explicit solution.  However, in the case of general magnetic fields it is no longer the case that the underlying system of trajectories for the potential-free Vlasov equation is explicitly solvable and therefore it is not clear whether this approach can be  carried out. So instead we follow the route taken by Lafleche in \cite{lafleche2021global} which is based on propagating the Eulerian moments, denoted $L_{n}(t)$, and defined by 
\begin{align*}
L_{n}(t)\coloneqq \int_{\mathbb{R}^{2d} }\left\vert x-t v\right\vert^{n}  f(t,x, v)\ dxdv. \end{align*}
In order to properly propagate $L_{n}(t)$ globally in time, we need both the decay assumption \eqref{eq:B}, the smallness assumption \eqref{smallness assumption} and the fast decay at infinity of the potential $K$. Note that the condition \eqref{smallness assumption} excludes the possibility of taking the initial data to be a probability density. This also explains why there is a substantial difference between the 2D and the 3D case. In 2D, the argument leading to global propagation of moments does not necessitate studying the evolution of $L_{n}(t)$, but only of $M_{n}(t)$. Due to proper cancellations, the magnetic field $B$ does not contribute any terms to the time derivative of $M_{n}(t)$. However, in 3D, this argument would yield only short time propagation of $M_{n}(t)$. Long time estimates require an evolution inequality for $L_{n}(t)$. But the time derivative of $L_{n}(t)$, unlike that of $M_{n}(t)$, includes terms contributed by $B$. To put it briefly, for what concerns the moments, in 2D the magnetic field is invisible, whereas it is visible in the 3D case, further complicating the analysis. Once propagation of moments is established, we can deduce propagation of regularity, leading to global existence of smooth solutions. This is the content of our second main theorem.
\begin{theorem}
Let the assumptions of Theorem \ref{main thm intro} hold and assume that $n>d( b '-1)$ where $\frac{1}{b}+\frac{1}{b'}=1$ and
\begin{align}
 b \coloneqq\begin{cases} 2,\ d=2,\\
\frac{3}{2},\ d=3.
\end{cases}
\label{definition of frakb}
\end{align}
Let $\phi_{0}$ be the constant from Theorem \ref{thm_bound_rho}, $\lambda(t)>0$ be a function such that $\lambda(t)(1+\left\Vert \nabla_x K\right\Vert_{ b ,\infty} \phi_{0})+\dot\lambda(t)\leq 0$ and let $w(x,v)\coloneqq 1+\left\vert x\right\vert^{2}+ \left\vert  v\right\vert^
{2}$. Suppose that $\nabla_{x, v}f^{\mathrm{in}}\in L^{p}(e^{w(x,v)})$ for some $p\in [2,\infty)$ and let $f$ be a classical solution of \eqref{MAGNETIC Vlasov Intro} on $[0,T]\times \mathbb{R}^{d}$ with initial data $f^{\mathrm{in}}$. Then, there is some $\overline{S}>0$ such that it holds that 
\begin{align*}
\left\Vert \nabla_{x}f(t,\cdot)\right\Vert_{L^{p}(e^{\lambda(t)w(x, v)})}^{p}+\left\Vert \nabla_{ v}f(t,\cdot)\right\Vert_{L^{p}(e^{\lambda(t)w(x,v)})}^{p}\leq \overline{S}\left(1+\left\Vert \nabla_{x}f^{\mathrm{in}}\right\Vert_{L^{p}(e^{w(x,v)})}^{
p}+\left\Vert \nabla_{ v}f^{\mathrm{in}}\right\Vert_{L^{p}(e^{w(x,v)})}^{p}  \right)^{e^{\overline{S}t}} \  \end{align*}
 for all $t\in [0,T]$. 
 \label{propagation of regularity intro}
\end{theorem}

Previous $W_1$-stability results for the Vlasov-Poisson equation, such as those of Holding and Miot \cite{Miot_2016, Holding_Miot} and Crippa et. al \cite{Crippa_co}, rely essentially on the second-order structure of the system. In contrast, Loeper's classical $W_2$-approach yields a weaker control, with solutions that are initially $\delta$-close remaining close only up to times of order $\log \abs{\log \delta}$. The rate of order $\sqrt{\abs{\log \delta}}$ obtained in $W_1$ for bounded macroscopic densities is regarded as \emph{optimal}. 

This rigidity was overcome by Iacobelli \cite{iacobelli_new_2022}, who introduced the \emph{kinetic Wasserstein distance}, a nonlinear modification of Loeper's functional that captures the anisotropy between position and velocity through the introduction of a time-dependent weight depending itself on the functional; $\pi_0$ is a $W_2$-optimal coupling between the initial solutions of \eqref{nonMAGNETIC Vlasov Intro} and $\lambda(t) \coloneqq \sqrt{\abs{\log D_2(t)}}$ for the functional
\begin{equation*}
    D_2(t) = \int_{(\XRd)^2} \lambda(t)\abs{X_1(t, x, v) - X_2(t, y, w)}^2 + \abs{V_1(t, x, v) - V_2(t, y, w)}^2 \: d\pi_0(x, v, y, w).
\end{equation*}
This nonlinear coupling between position and velocity allows the derivation of sharp stability estimates solely based on first-order structure, and recovers the optimal rate $\sqrt{\abs{\log \delta}}$ in the stronger $W_2$-distance, thereby improving substantially the quasi-neutral limit \cite{iacobelli_new_2022}. In the quantum setting, together with Lafleche \cite{iacobelli2024enhanced}, they enhanced the convergence from the Hartree to the Vlasov--Poisson equation in the semi-classical analog of the 2-Wasserstein distance introduced by Golse, Mouhot and Paul \cite{golse2017schrodinger}. The kinetic Wasserstein framework was later also applied to the magnetized Vlasov--Poisson equation by Rege \cite{rege2025stability}, and with Junn\'e \cite{junne_stability_2025}, extended to Yudovich macroscopic densities -- the largest class for which uniqueness was obtained by the second-order approach of \cite{Crippa_co}.

However, the $W_2$-stability estimates obtained so far in the magnetized case do not reach this optimal order. The key observation underlying our improvement is that the external magnetic field does not produce work, which allows control of the velocity magnitude independently of $B(t, x)$ while its direction is still influenced by the presence of that field through 
\begin{align}
\begin{cases}
\begin{array}{cc}
\dot X (t,x,v)= V(t,x,v), & X(0,x,v)=x,\\
\dot V(t,x,v)= -\nabla_x K \star \rho_{f}(t, X(t,x,v)) - V(t,x,v) \wedge B(t, x), & V(0,x,v)=v. \label{trajectories of Vlasov mag}
\end{array}\end{cases}    
\end{align}
By incorporating this no-work property into the kinetic Wasserstein distance framework, we refine the existing estimates and recover the optimal $\sqrt{\abs{\log \delta}}$ rate for the magnetized Vlasov--Poisson equation. This constitutes our third theorem.

\begin{theorem}\label{thm:Iacobelli's W_p}
    Let $f_1, f_2 \ge 0$ be two weak solutions to \eqref{MAGNETIC Vlasov Intro} on $[0, T)$ with respective densities $\rho_{f_i} \coloneqq \int f_i \: dv$, $i = 1,2,$ of common mass $1$, and let
    \begin{align*}
        K(x)\coloneqq \begin{cases}
\begin{array}{cc}
\pm \frac{1}{2\pi}\log(x) & d=2,\\
\pm \frac{e^{-\kappa\left\vert x\right\vert }}{4\pi\left\vert x\right\vert} & d=3. 
\end{array}\ \ \ \  ,\ \kappa\ge0.\end{cases}
    \end{align*}
    Let $1 \le p < +\infty$, and set
    \begin{equation*}\label{eq:A Iacobelli}
        A(t) \coloneqq \widetilde A(t) + \int_0^t \widetilde A(\tau) \: d\tau, \quad \widetilde A(t) \coloneqq \max\left\{1, \left[\norm{\rho_{f_2}} + \norm{\rho_{f_1}}^{\frac{1}{p}} \max\left\{\norm{\rho_{f_1}}, \norm{\rho_{f_2}}\right\}^{\frac{1}{p'}} \right]_{\infty}(t)\right\},
    \end{equation*}
    which is assumed to be in $L^1([0, T))$. Assume that $B(t, \cdot)$ is uniformly Lipschitz on $[0, T)$ and that initially one of the two solutions satisfies $M_k(0) \le (C_0 k)^k$ for all $k > 0$ and some constant $C_0 > 0$. Then there is a universal constant $c_0 > 0$ and a constant $C_{p,B} > 0$ that depends only on $p, C_0$ and $\norm{\nabla_x B}_{\infty}$ such that if $W_p^p(f_1^{\mathrm{in}}, f_2^{\mathrm{in}})$ is sufficiently small so that $W_p^p(f_1^{\mathrm{in}}, f_2^{\mathrm{in}}) \le c_0$ and
    \begin{equation*}
        \sqrt{\abs{\log \left(W_p^p\left(f_1^{\mathrm{in}}, f_2^{\mathrm{in}}\right) \sqrt{\abs{\log W_p^p\left(f_1^{\mathrm{in}}, f_2^{\mathrm{in}}\right)}}^p\right)}} \ge C_{p,B}\int_0^t A(\tau) \: d\tau + 1,  
    \end{equation*}
    then
    \begin{equation*}
            W_p^p\left(f_1(t), f_2(t)\right) \le \exp\left\{-\left(\sqrt{\abs{\log\Bigg\{W_p^p\left(f_1^{\mathrm{in}}, f_2^{\mathrm{in}}\right)\sqrt{\abs{\log W_p^p\left(f_1^{\mathrm{in}}, f_2^{\mathrm{in}}\right)}}^p\Bigg\}}} - C_{p,B}\int_0^t A(\tau) \: d\tau\right)^2\right\}.
    \end{equation*}
\end{theorem}

The paper is organized as follows. In \S{\ref{prop of moments sec}} we prove global propagation of velocity moments, first in 2D and then in 3D. The 3D case requires a more careful analysis due to the fact that it necessitates studying the propagation of the Eulerian moments as well in order to overcome a potential Riccati type blow up. In \S{\ref{sec prop of regularity}} we prove that propagation of regularity is implied from propagation of moments. Finally, in \S{\ref{sec_stability_estimate}}  we improve results about the stability of solutions using the kinetic Wasserstein distance.

\section{Propagation of moments}\label{prop of moments sec}
\subsection{Preliminaries} 
We start by presenting several general estimates which will be central to the forthcoming estimates. We follow closely \cite{lafleche2019propagation}. For any function $f:[0,T]\times \mathbb{R}^{d}\times\mathbb{R}^{d}\rightarrow \mathbb{R}_{+}$ let us introduce the velocity moments, Eulerian moments and space moments denoted respectively $M_{n,f},L_{n,f},N_{n,f}$ and defined 
by
\begin{align*}
&M_{n,f}(t)\coloneqq \int_{\mathbb{R}^{2d}}\left| v\right|^{n}f(t,x, v)\ dxd v, \\
&L_{n,f}(t)\coloneqq \int_{\mathbb{R}^{2d}}\left|tv-x\right|^{n}f(t,x, v)\ dxd v,\\
& N_{n,f}(t)\coloneqq \int_{\mathbb{R}^{2d}}\left\vert x\right\vert^{n}f(t,x, v)\ dxd v. 
\end{align*}
We also introduce the partial  velocity moments  denoted $\rho_{n,f}(t,x)$
and defined by
\begin{align*}
\rho_{n,f}(t,x)\coloneqq\int_{\mathbb{R}^{d}}\left|v\right|^{n}f(t,x,v)\ dv.  \end{align*}
Note that  $\rho_{f}=\rho_{0,f}$, where $\rho_{f}$ is the density. Whenever there is no ambiguity the subscript $f$ will be omitted. 
\begin{lemma}
\label{Magnetic-kinetic-interpolation} \textup{(Kinetic
interpolation inequality)}. Suppose that $d\geq 1$, that $0\leq k\leq n$
and put $p_{n,k} \coloneqq \frac{n+d}{k+d}, \theta_{n,k}\coloneqq\frac{1}{p'_{n,k}}$. 
Then, for any $f\in L^{\infty}(\mathbb{R}^{d}\times\mathbb{R}^{d})$ with $M_{n,f} < \infty$
it holds that 
\begin{align*}
\left\Vert \int_{\mathbb{R}^{d}}\left| v\right|^{k}f(\cdot, v)\ d v\right\Vert _{L^{p_{n,k}}_{x}}\leq C\left(\int_{\mathbb{R}^{2d}}\left| v\right|^{n}f(x, v)\ dxd v\right)^{1-\theta_{n,k}}\left\Vert f\right\Vert _{L_{x,v}^\infty}^{\theta_{n,k}},
\end{align*}
where $C=C_{d,k}=\frac{d+k+1}{d+k}$. In particular, 
\[
\left\Vert \rho_{f}\right\Vert _{L_x^{p_{n}}}\leq C\left(\int_{\mathbb{R}^{2d}}\left| v\right|^{n}f(x, v)\ dxd v\right)^{1-\theta_{n}}\left\Vert f\right\Vert _{L_{x,v}^\infty}^{\theta_{n}},
\]
where $p_{n}\coloneqq p_{n,0}=\frac{n+d}{d}$ and $\theta_{n}\coloneqq\frac{1}{p'_{n}}$. 
\end{lemma}
We also recall the weak Young inequality (see for instance \cite{robinson2016three}).  
\begin{theorem} 
Let $(p,q,r)\in(1,\infty)^{3}$ with $1+\frac{1}{p}=\frac{1}{q}+\frac{1}{r}$.
Then there exists $C=C(p,q,r)<1$ such that 
\[
\left\Vert f\star g\right\Vert _{p}\leq C\left\Vert f\right\Vert _{q,\infty}\left\Vert g\right\Vert _{r}.
\] \label{weak Young}
\end{theorem}
Recall that conservation of $L^{p}$ norms is ensured, i.e., we have the following proposition, the proof of which is standard.  
\begin{proposition}
\label{conservation of Lp} Let $f^{\mathrm{in}}\in L^{p}(\mathbb{R}^{d}\times \mathbb{R}^{d})$ for some $1\leq p \leq \infty$ and let $f$ be a classical solution to \eqref{MAGNETIC Vlasov Intro} with initial data $f^{\mathrm{in}}$. Then it holds that 
\begin{align*}
\left\Vert f(t,\cdot)\right\Vert_{p}=\left\Vert f^\mathrm{in}\right\Vert_{p} \ \mbox{for all}\ t\in[0,T].    
\end{align*}
\end{proposition}
The following estimate will be employed in the 2D case to get global propagation of moments and in the 3D case to get local propagation of moments. 
\begin{theorem}
Let $d=2,3$ and let $K$ be given by \eqref{Coulomb def}. Let $f$ be a classical solution of  \eqref{MAGNETIC Vlasov Intro} with initial data $0\leq f^{\mathrm{in}}\in L^{1}(\mathbb{R}^{d}\times \mathbb{R}^{d})\cap L^{\infty}(\mathbb{R}^{d}\times \mathbb{R}^{d})$. Then, it holds that 
\begin{align}
\frac{d}{dt}M_{n}(t)\leq  C_{\alpha,d}n\left\Vert \nabla_x K\right\Vert_{ b ,\infty} \left\Vert \rho_{n-1}(t,\cdot)\right\Vert _{\alpha}\left\Vert \rho_{f}(t,\cdot)\right\Vert _{\beta} \ \mbox{for all}\ t\in [0,T],
\label{def of b}
\end{align}
where $C_{\alpha,d}<1$,  
$$ b =\begin{cases} 2,\ d=2\\
\frac{3}{2},\ d=3
\end{cases}$$
and $\alpha$,$\beta$ are such that \begin{align*}
\frac{1}{\alpha}+\frac{1}{\alpha'}=1,\ \frac{1}{\alpha'}+\frac{1}{\beta'}=\frac{1}{ b }.      
\end{align*}
\label{diff ine general d}
\end{theorem}
\begin{proof}
We compute 
\begin{align}
\frac{d}{dt}M_{n}(t)&=-\int_{\mathbb{R}^{2d}}\left| v\right|^{n}\left(v\cdot\nabla_{x}f(t,x,v)-(\nabla_{x}K\star\rho_{f}(t,x)+v\wedge B(t,x))\cdot \nabla_{v}f(t,x,v)\right)\ dxd v \notag\\
&=-\int_{\mathbb{R}^{2d}}\left\vert v\right\vert^{n}v\cdot \nabla_{x}f(t,x,v)\ dxdv+\int_{\mathbb{R}^{2d}}\left\vert v\right\vert^{n}\nabla_{x}K\star \rho_{f}(t,x)\cdot \nabla_{v}f(t,x,v)\ dxdv \notag\\
&+\int_{\mathbb{R}^{2d}}\left\vert v\right\vert^{n}v\wedge B(t,x)\cdot \nabla_{v}f(t,x,v)\ dxdv.   \label{computation of time derivative Mn}  
\end{align}
Integration by parts of the first integral in the right-hand side of \eqref{computation of time derivative Mn} reveals that it is equal to $0$. 
Furthermore integration by parts of the third integral shows that it is $0$ because 
\begin{align*}
-\int_{\mathbb{R}^{2d}}\mathrm{div}_{v}(\left\vert v\right\vert^{n}v\wedge B(t,x))f(t,x,v)\ dxdv=&-n\int_{\mathbb{R}^{2d}}\left\vert v\right\vert^{n-2}v\cdot v \wedge B(t,x)f(t,x,v)\ dxdv\\
&-\int_{\mathbb{R}^{2d}}\left\vert v\right\vert^{n}\mathrm{div}_{v}(v \wedge B(t,x))f(t,x,v)\ dxdv=0,
\end{align*}
where in the last equality we made use of the observation that $v$ is orthogonal to $v\wedge B$ and that $\mathrm{div}_v(v \wedge B)=0$. Finally, integrating by parts the second term in the right-hand side of \eqref{computation of time derivative Mn} we see it is equal to  
\begin{align*}
-n\int_{\mathbb{R}^{2d}}\left\vert v\right\vert^{n-2}v \cdot \nabla_{x}K\star \rho_{f}(t,x)f(t,x,v)\ dxdv.     
\end{align*}
To conclude, we obtain the identity 
\begin{align*}
\frac{d}{dt}M_{n}(t)=-n\int_{\mathbb{R}^{2d}}\left\vert v\right\vert^{n-2}v\cdot \nabla_{x}K\star \rho_{f}(t,x)f(t,x,v)\ dxdv.   
\end{align*}
Therefore, we deduce by using Hölder inequality that 
\begin{align*} 
\frac{d}{dt}M_{n}(t)\leq &n\int_{\mathbb{R}^{d}}\rho_{n-1}(t,x)\left|\nabla_x K\star\rho_{f}\right|(t,x)\ dx\\
\leq& n\left\Vert \rho_{n-1}(t,\cdot)\right\Vert _{\alpha}\left\Vert \nabla_x K\star\rho_{f}(t,\cdot)\right\Vert _{\alpha'}\leq C_{\alpha,d}n\left\Vert \nabla_x K\right\Vert_{ b ,\infty} \left\Vert \rho_{n-1}(t,\cdot)\right\Vert _{\alpha}\left\Vert \rho_{f}(t,\cdot)\right\Vert _{\beta},\label{eq:-4}
\end{align*}
where in the last inequality we applied Theorem \ref{weak Young} with the choice  $(\alpha,\beta)\in(1,\infty)^{2}$: 
\begin{equation*}
\frac{1}{\alpha}+\frac{1}{\alpha'}=1,\ 1+\frac{1}{\alpha'}=\frac{1}{ b }+\frac{1}{\beta}, \label{eq:-8}
\end{equation*}
or equivalently 
\begin{equation*}
\frac{1}{\alpha}+\frac{1}{\alpha'}=1,\ \frac{1}{\alpha'}+\frac{1}{\beta'}=\frac{1}{ b }.\label{eq:-11}
\end{equation*}    
\end{proof}
\subsection{The case $d=2$}
In the $2D$ case, $M_{n}(t)$ is governed by a sub-linear differential inequality, which enables to get global propagation of moments -- this is however not the case for 3D, which requires additional considerations, which will be presented in the next section. The following theorem proves point i. in Theorem \ref{main thm intro}. 

\begin{theorem}
Let the assumptions of Theorem \ref{main thm intro} hold.
Then, there is some function $\Phi_{n}(F,M, t)$, continuous in $t$, such that it holds that 

\[
M_{n}(t)\leq\Phi_{n} \left(\left\Vert f^{\mathrm{in}}\right\Vert_{L^\infty\cap L^{1}},\mathbf{M}_{n}^{\mathrm{in}}, t \right)\ \mbox{for all}\ t\in [0,T]. 
\]
\label{propagation of magnetic moments}
\end{theorem}
\begin{proof}
The proof is by induction on $n$, i.e., the induction hypothesis is that there exists a function $\Phi_{n-1}(F,M,t)$, continuous in $t$ and increasing in $M$, such that 
\begin{align*}
M_{n-1}(t)\leq \Phi_{n-1}\left(\left\Vert f^{\mathrm{in}}\right\Vert_{\infty},\mathbf{M}^{in}_{n-1},t \right)\  \mbox{for all}\ t\in [0,T].     
\end{align*}
By Theorem \ref{diff ine general d} it holds that 
\begin{align}
\frac{d}{dt}M_{n}(t)&\leq n\left\Vert \nabla_x K\right\Vert_{2,\infty} \left\Vert \rho_{n-1}(t,\cdot)\right\Vert _{\alpha_n}\left\Vert \rho_{f}(t,\cdot)\right\Vert _{\beta_n}, \label{reminder}
\end{align}
with the choice  
\begin{equation*}
\frac{1}{\alpha_n}+\frac{1}{\alpha_n'}=1,\ \frac{1}{\alpha_n'}+\frac{1}{\beta_n'}=\frac{1}{2}.
\end{equation*}
By Lemma \ref{Magnetic-kinetic-interpolation}, if we take $\alpha'_{n}=p'_{n,n-1}=n+2$
and $\theta_{n,n-1}=\frac{1}{\alpha'_{n}}$ we get 
\begin{equation}
\left\Vert \rho_{n-1}(t,\cdot)\right\Vert _{\alpha_{n}} \leq \frac{n+2}{n+1}M_{n}^{1-\theta_{n,n-1}}(t)\left\Vert f(t,\cdot)\right\Vert _{\infty}^{\theta_{n,n-1}}= \frac{n+2}{n+1}M_{n}^{1-\theta_{n,n-1}}(t)\left\Vert f^{\mathrm{in}}\right\Vert _{\infty}^{\theta_{n,n-1}},\label{eq:-12}
\end{equation}
where the last equality is because $L^{\infty}$ norms are conserved by Proposition \ref{conservation of Lp}. We now calculate all the exponents involved.
\[\left\{ \begin{array}{lc}
\alpha'_{n}=n+2\Longrightarrow\alpha
_{n}=\frac{n+2}{n+1}\\
\frac{1}{\alpha_{n}'}+\frac{1}{\beta_{n}'}=\frac{1}{2}\Longrightarrow\beta_{n}'=\frac{2(n+2)}{n}\Longrightarrow\beta_{n}=\frac{2n+4}{n+4}\\
p'_{n-1}=\frac{n+1}{n-1}\Longrightarrow p_{n-1}=\frac{n+1}{2},\  p'_{n}=\frac{n+2}{n}\Longrightarrow p_{n}=\frac{n+2}{2}. 
\end{array}\right.\]
\textbf{Step 1}. \textit{The base case n=1}. For $n=1$ we compute that $\beta_{1}=\frac{6}{5}\leq \frac{3}{2}= p_{1}$ and $\theta_{1}=\frac{1}{p'_{1}}=\frac{1}{3}$. Observe also that for the choice $\lambda=\frac{1}{2}$ we have the relation 
\begin{align*}
\frac{1}{\frac{6}{5}}=\frac{\lambda}{\frac{3}{2}}+\frac{1-\lambda}{1}.    
\end{align*}
Therefore, it follows by Lebesgue's interpolation and by Lemma \ref{Magnetic-kinetic-interpolation} that  
\begin{align}
\left\Vert \rho_{f}(t,\cdot)\right\Vert_{\beta_{1}}= \left\Vert \rho_{f}(t,\cdot)\right\Vert_{\frac{6}{5}}\leq \left\Vert \rho_{f}(t,\cdot)\right\Vert_{\frac{3}{2}}^{\frac{1}{2}}\left\Vert \rho_{f}(t,\cdot)\right\Vert_{1}^{\frac{1}{2}}\leq \left(\frac{3}{2}\right)^{\frac{1}{2}}M_{1}^{\frac{1}{3}}(t)\left\Vert f^{\mathrm{in}}\right\Vert_{\infty}^{\frac{1}{6}} \left\Vert f^{\mathrm{in}}\right\Vert_{1}^{\frac{1}{2}}. \label{betaest}     
\end{align}
In addition one directly calculates that $\alpha_{1}=\frac{3}{2}$ and $\theta_{1}=\frac{1}{3}$. Inserting this inside \eqref{eq:-12} yields the inequality 
\begin{align}
\left\Vert \rho_{f}(t,\cdot)\right\Vert_{\alpha_{1}}=\left\Vert \rho_{f}(t,\cdot)\right\Vert_{\frac{3}{2}}\leq \frac{3}{2}M_{1}^{\frac{2}{3}}(t)\left\Vert f^{\mathrm{in}}\right\Vert_{\infty}^{\frac{1}{3}}. \label{alpahest}       
\end{align}
Substituting \eqref{betaest} and \eqref{alpahest} inside \eqref{reminder} we obtain  \begin{align*}
\frac{d}{dt}M_{1}(t) &\leq \left(\frac{3}{2}\right)^{\frac{3}{2}}\left\Vert \nabla_x K\right\Vert_{2,\infty}\left\Vert f^{\mathrm{in}}\right\Vert_{\infty}^{\frac{2}{3}}\left\Vert f^{\mathrm{in}}\right\Vert_{1}^{\frac{1}{2}}M_{1}(t).         
\end{align*}
The result now follows by Gr\"onwall's lemma.\\ 
\textbf{Step 2}. \textit{General n}. Note that $1\leq \beta_{n}\leq p_{n-1}$ for all $n\geq 2$ (and therefore $p'_{n-1}\leq \beta'_{n}$). Thus, observing the identity 
\begin{align*}
\frac{1}{\beta_{n}}=\frac{\frac{p_{n-1}'}{\beta'_{n}}}{p_{n-1}}+\frac{1-\frac{p'_{n-1}}{\beta'_{n}}}{1}    
\end{align*}
we can apply Lebesgue interpolation and Lemma \ref{Magnetic-kinetic-interpolation} in order to find 
\begin{align}
\left\Vert \rho_{f}(t,\cdot)\right\Vert _{\beta_{n}}\leq&\left\Vert \rho_{f}(t,\cdot)\right\Vert _{p_{n-1}}^{\frac{p'_{n-1}}{\beta'_{n}}}\left\Vert \rho_{f}(t,\cdot)\right\Vert _{1}^{1-\frac{p'_{n-1}}{\beta'_{n}}}\leq \left(\frac{3}{2}\right)^{\frac{p'_{n-1}}{\beta'_{n}}}M_{n-1}^{(1-\theta_{n-1})\frac{p'_{n-1}}{\beta'_{n}}}(t)\left\Vert f(t,\cdot)\right\Vert _{\infty}^{\frac{p'_{n-1}\theta_{n-1}}{\beta'_{n}}}\left\Vert f(t,\cdot)\right\Vert_{1}^{1-\frac{p'_{n-1}}{\beta_{n}'}} \notag \\=&\left(\frac{3}{2}\right)^{\frac{p'_{n-`1}}{\beta_{n}'}}M_{n-1}^{\left(1-\theta_{n-1}\right)\frac{p_{n-1}'}{\beta_{n}'}}(t)\left\Vert f^{\mathrm{in}}\right\Vert _{\infty}^{\frac{1}{\beta'_{n}}}\left\Vert f^{\mathrm{in}}\right\Vert_{1}^{1-\frac{p'_{n-1}}{\beta_{n}'}}. 
\label{eq:-13}
\end{align}
Therefore, using inequalities \eqref{eq:-12}
and \eqref{eq:-13} in \eqref{reminder}, we find that 
\begin{equation}
\left|\frac{d}{dt}M_{n}(t)\right|\leq C_{n}M_{n-1}^{\left(1-\theta_{n-1}\right)\frac{p_{n-1}'}{\beta_{n}'}}(t)M_{n}^{1-\theta_{n,n-1}}(t)\left\Vert f^{\mathrm{in}}\right\Vert_{\infty}^{\theta_{n,n-1}+\frac{1}{\beta'_{n}}}\left\Vert f^{\mathrm{in}}\right\Vert_{1}^{1-\frac{p'_{n-1}}{\beta_{n}'}},\label{eq:-5}
\end{equation}
where the constant $C_{n}$ is given by  $C_{n}=\left(\frac{3}{2}\right)^{\frac{p'_{n-1}}{\beta'_{n}}}\frac{n(n+2)}{(n+1)}\left\Vert \nabla_x K\right\Vert_{2,\infty} >0$. 
By induction hypothesis 
there exists a  function $\Phi_{n-1}(F,M,t)$
continuous in $t$ and increasing in $M$ such that 
\begin{align*}
M_{n-1}(t)\leq\Phi_{n-1} \left(\left\Vert f^{\mathrm{in}}\right\Vert_{L^\infty\cap L^{1}},\mathbf{M}^{\mathrm{in}}_{n-1},t \right)\ \mbox{for all} \ t \in [0,T].     
\end{align*}
Therefore, since $1-\theta_{n,n-1} < 1$, Gr\"onwall's lemma and inequality \eqref{eq:-5} yield the induction step. 
\end{proof}
\subsection{The case $d=3$}
This section is devoted to the proof of point ii) in Theorem \ref{main thm intro}. We start by considering propagation of moments on some short time interval. 
\subsubsection{Short time propagation of moments of order $n=3,4$}
\begin{theorem}
\label{3d estZn}
Suppose that $n=3,4$. Let $d=3$ and let $K(x)=\pm \frac{e^{-\left\vert x\right\vert }}{\left\vert x\right\vert}$ be the screened Coulomb potential. Let $f$ be a classical solution of \eqref{MAGNETIC Vlasov Intro}  with $0\leq f^{\mathrm{in}}\in L^{1}(\mathbb{R}^{3}\times \mathbb{R}^{3})\cap L^{\infty}(\mathbb{R}^{3}\times\mathbb{R}^{3})$. Suppose that $M_{n}(0)\leq \mathbf{M}_{n}^{\mathrm{in}}$ for $n=3,4$.
Set $T_{\ast}=\min\left\{\frac{5}{16\left\Vert f^{\mathrm{in}}\right\Vert_{\infty}^{\frac{2}{3}} M_{3}^{\frac{1}{3}}(0)},\frac{9}{8\left\Vert f^{\mathrm{in}}\right\Vert^{\frac{21}{23}}_{\infty} M_{4}^{\frac{1}{7}}(0)M_{3}^{\frac{1}{3}}(0)}\right\}$. Then for all $t\in [0,T_{\ast}]$ it holds that 
\begin{align*}
M_{3}(t)\leq 2^{\frac{3}{2}}M_{3}(0), \ M_{4}(t)\leq 2^{7}M_{4}(0).    
\end{align*}
\end{theorem}
\begin{proof}
By Theorem \ref{diff ine general d} it holds that 
\begin{align}
\label{diff_ineq_d3}
\frac{d}{dt}M_{n}(t)&\leq n\left\Vert \nabla_x K\right\Vert_{\frac{3}{2},\infty} \left\Vert \rho_{n-1}(t,\cdot)\right\Vert _{\alpha_n}\left\Vert \rho_{f}(t,\cdot)\right\Vert _{\beta_n},
\end{align}
with the choice  
\begin{equation*}
\frac{1}{\alpha_n}+\frac{1}{\alpha_n'}=1,\ \frac{1}{\alpha_n'}+\frac{1}{\beta_n'}=\frac{2}{3}.
\end{equation*} 
By Lemma \ref{Magnetic-kinetic-interpolation}, if we take $\alpha_{n}'=p'_{n,n-1}=n+3$ and $\theta_{n,n-1}=\frac{1}{\alpha_{n}'}$ we get  
\begin{align}
\left\Vert \rho_{n-1}(t,\cdot)\right\Vert _{\alpha_{n}}\leq \frac{n+3}{n+2}M_{n}^{1-\theta_{n,n-1}}(t)\left\Vert f(t,\cdot)\right\Vert _{\infty}^{\theta_{n,n-1}}&= \frac{n+3}{n+2}M_{n}^{1-\theta_{n,n-1}}(t)\left\Vert f^{\mathrm{in}}\right\Vert _{\infty}^{\theta_{n,n-1}}.\label{eq:-12d=3}
\end{align}
We now calculate all the exponents involved.
\begin{align}
\left\{ \begin{array}{lc}
\alpha'_{n}=n+3\Longrightarrow\alpha
_{n}=\frac{n+3}{n+2}\\
\frac{1}{\alpha_{n}'}+\frac{1}{\beta_{n}'}=\frac{2}{3}\Longrightarrow\beta_{n}'=\frac{3n+9}{2n+3}\Longrightarrow\beta_{n}=\frac{3n+9}{n+6}\\
p'_{n-1}=\frac{n+2}{n-1}\Longrightarrow p_{n-1}=\frac{n+2}{3},\  p'_{n}=\frac{n+3}{n}\Longrightarrow p_{n}=\frac{n+3}{3}. 
\end{array}\right.  \label{exponents}  
\end{align}
In the case  $n=3,4$ we have that 
$p_{n-1}\leq\beta_{n}\leq p_{n}$ so that 
$p_{n}'\leq \beta_{n}'\leq p'_{n-1}$. 
Therefore, there exists $\varepsilon_{n} \in (0,1]$ such that 
\begin{align}
\frac{1}{\beta_{n}'}=\frac{\varepsilon_{n}}{p_{n}'}+\frac{1-\varepsilon_{n}}{p_{n-1}'},\label{eq:-9conj} \end{align}
which, as one readily checks, is equivalent to 
\begin{align}
\frac{1}{\beta_{n}}=\frac{\varepsilon_{n}}{p_{n}}+\frac{1-\varepsilon_{n}}{p_{n-1}}.\label{eq:-9}    
\end{align}
Thus, by Lebesgue interpolation and Lemma \ref{Magnetic-kinetic-interpolation}
we get the following inequality 
\begin{align}
\left\Vert \rho_{f}(t,\cdot)\right\Vert _{\beta_{n}}\leq\left\Vert \rho_{f}(t,\cdot)\right\Vert _{p_{n}}^{\varepsilon_{n}}\left\Vert \rho_{f}(t,\cdot)\right\Vert _{p_{n-1}}^{1-\varepsilon_{n}}&\leq \frac{4}{3}M_{n-1}^{(1-\varepsilon_{n})(1-\theta_{n-1})}(t)M_{n}^{\varepsilon_{n}(1-\theta_{n})}(t)\left\Vert f^{\mathrm{in}}\right\Vert _{\infty}^{(1-\varepsilon_{n})\theta_{n-1}+\varepsilon_{n}\theta_{n}}, 
\label{rhof beta norm} 
\end{align}
where we recall that $\theta_{n}=\frac{1}{p_{n}'}$. By  \eqref{eq:-12d=3} and \eqref{rhof beta norm} it follows that 
\begin{align}
&\left\Vert \rho_{n-1}(t,\cdot)\right\Vert _{\alpha_{n}}\left\Vert \rho_{f}(t,\cdot)\right\Vert _{\beta_{n}}\notag \\&\leq\frac{4(n+3)}{3(n+2)} M_{n-1}^{(1-\varepsilon_{n})(1-\theta_{n-1})}(t)M_{n}^{\varepsilon_{n}(1-\theta_{n})}(t)\left\Vert f^{\mathrm{in}}\right\Vert _{\infty}^{(1-\varepsilon_{n})\theta_{n-1}+\varepsilon_{n}\theta_{n}}M_{n}^{1-\theta_{n,n-1}}(t)\left\Vert f^{\mathrm{in}}\right\Vert _{\infty}^{\theta_{n,n-1}}\notag\\
&=\frac{4(n+3)}{3(n+2)}\left\Vert f^{\mathrm{in}}\right\Vert _{\infty}^{(1-\varepsilon_{n})\theta_{n-1}+\varepsilon_{n}\theta_{n}+\theta_{n,n-1}}M_{n-1}^{\Theta_{0,n}}(t)M_{n}^{\Theta_{n}}(t) \label{first ine}
\end{align}
where we have set 
\[
\Theta_{0,n}\coloneqq(1-\varepsilon_{n})(1-\theta_{n-1}),\ \Theta_{n}\coloneqq1-\theta_{n,n-1}+\varepsilon_{n}(1-\theta_{n}).
\]
We continue by calculating $\varepsilon_{n}$ from \eqref{eq:-9conj}. Using \eqref{exponents} we see that  
\begin{align*}
\varepsilon_{n}=\left(\frac{1}{\beta_{n}'}-\frac{1}{p_{n-1}'}\right)\left(\frac{1}{p_{n}'}-\frac{1}{p_{n-1}'}\right)^{-1}= \left(\frac{2n+3}{3n+9}-\frac{n-1}{n+2}\right)\left(\frac{n}{n+3}-\frac{n-1}{n+2}\right)^{-1}.    
\end{align*}
If $n=3$ then we
get 
\begin{align*}
\varepsilon_{3}=\left(\frac{1}{2}-\frac{2}{5}\right)\left(\frac{1}{2}-\frac{2}{5}\right)^{-1}=1     
\end{align*}
while if $n=4$ then 
\begin{align*}
\varepsilon_{4}=\left(\frac{11}{21}-\frac{1}{2}\right)\left(\frac{4}{7}-\frac{1}{2}\right)^{-1}=\frac{14}{42}=\frac{1}{3}.     
\end{align*}
In the case $n=3$ we compute that $\theta_{3} = \frac{1}{2}$ and $\theta_{3,2}=\frac{1}{6}$ so  
\begin{align*}
\Theta_{3}=\frac{5}{6}+\frac{1}{2}=\frac{4}{3},    \ \Theta_{0,3}=0 
\end{align*}
while in the case $n=4$ we compute that $\theta_{4}=\frac{4}{7}$ and $\theta_{4,3}=\frac{1}{7}$ so 
\begin{align*}
\Theta_{4}=\frac{6}{7}+\frac{1}{3}\left(1-\frac{4}{7}\right)=1, \ \Theta_{0,4}=\frac{1}{3}.
\end{align*}
In both cases, plugging \eqref{first ine} inside \eqref{diff_ineq_d3} implies 
\begin{align*}
\frac{d}{dt}M_{n}(t)&\leq n\left\Vert \nabla_x K\right\Vert_{\frac{3}{2},\infty} \frac{4(n+3)}{3(n+2)}\left\Vert f^{\mathrm{in}}\right\Vert _{\infty}^{(1-\varepsilon_{n})\theta_{n-1}+\varepsilon_{n}\theta_{n}+\theta_{n,n-1}}M_{n-1}^{\Theta_{0,n}}(t)M_{n}^{\Theta_{n}}(t).  
\end{align*} 
For $n=3$ , plugging in the values of the parameters yields the inequality 
\begin{align*}
\frac{d}{dt}M_{3}(t)\leq\frac{24}{5}\left\Vert f^{\mathrm{in}}\right\Vert_{\infty}^{\frac{2}{3}}M_{3}^{\frac{4}{3}}(t).      
\end{align*}
Solving the above inequality yields 
\begin{align*}
M_{3}^{-\frac{1}{3}}(0)-M_{3}^{-\frac{1}{3}}(t)\leq \frac{8}{5}\left\Vert f^{\mathrm{in}}\right\Vert_{\infty}^{\frac{2}{3}}t, 
\end{align*}
which entails 
\begin{align*}
M_{3}(t)\leq \left(\frac{1}{M_{3}^{-\frac{1}{3}}(0)-\frac{8}{5}\left\Vert f^{\mathrm{in}}\right\Vert_{\infty}^{\frac{2}{3}}t}\right)^{3}.      
\end{align*}
Thus given $P>0$ we have  
$T^{\ast}_{3}\geq \frac{5(M^{-\frac{1}{3}}_{3}(0)-P)}{8\left\Vert f^{\mathrm{in}}\right\Vert_{\infty}^{\frac{2}{3}}}$. Choosing 
$P=\frac{1}{2}M_{3}^{-\frac{1}{3}}(0)$ we conclude that 
\begin{align*}
T_{3}^{\ast}\geq \frac{5}{16\left\Vert f^{\mathrm{in}}\right\Vert_{\infty}^{\frac{2}{3}} M_{3}^{\frac{1}{3}}(0)}    
\end{align*}
and 
\begin{align}
t \in [0,T_3^{\ast}]\Rightarrow M_{3}(t)\leq 2^{3}M_{3}(0). \label{Z3 est}   
\end{align}
For $n=4$, plugging the values of the parameters 
and utilizing \eqref{Z3 est} yields  
\begin{align*}
\frac{d}{dt}M_{4}(t)\leq \frac{14}{9}\left\Vert f^{\mathrm{in}}\right\Vert_{\infty}^{\frac{23}{21}}M_{3}^{\frac{1}{3}}(t)M_{4}^{\frac{8}{7}}(t) \leq \frac{28}{9}\left\Vert f^{\mathrm{in}}\right\Vert_{\infty}^{\frac{23}{21}}M_{3}^{\frac{1}{3}}(0)M_{4}^{\frac{8}{7}}(t),    
\end{align*}
which can be written as 
\begin{align*}
 -7\frac{d}{dt}M_{4}^{-\frac{1}{7}}(t)\leq \frac{28}{9}\left\Vert f^{\mathrm{in}}\right\Vert_{\infty}^{\frac{21}{23}}M_{3}^{\frac{1}{3}}(0)\Rightarrow-\frac{d}{dt}M_{4}^{-\frac{1}{7}}(t)\leq \frac{4}{9}\left\Vert f^{\mathrm{in}}\right\Vert_{\infty}^{\frac{21}{23}}M_{3}^{\frac{1}{3}}(0).    
\end{align*}
Solving the above inequality yields
\begin{align*}
M_{4}(t)\leq \frac{1}{\left(M_{4}^{-\frac{1}{7}}(0)-\frac{4}{9}\left\Vert f^{\mathrm{in}}\right\Vert_{\infty}^{\frac{21}{23}}M_{3}^{\frac{1}{3}}(0)t\right)^{7}}.     
\end{align*}
Thus, if we take   $T_{\ast}=\min\{T_{3}^{\ast},T_{4}^{\ast}\}=\min\left\{\frac{5}{16\left\Vert f^{\mathrm{in}}\right\Vert_{\infty}^{\frac{2}{3}} M_{3}^{\frac{1}{3}}(0)},\frac{9}{8\left\Vert f^{\mathrm{in}}\right\Vert^{\frac{21}{23}}_{\infty} M_{4}^{\frac{1}{7}}(0)M_{3}^{\frac{1}{3}}(0)}\right\}$
we can ensure that for all $t\in [0,T_{\ast}]$ it holds that 
\begin{align*}
M_{3}(t)\leq 2^{3}M_{3}(0), \ M_{4}(t)\leq 2^{7}M_{4}(0).    
\end{align*}
\end{proof}
\begin{remark}
The short time estimate can be obtained in the same manner illustrated above when $K$ is the Coulomb potential. The use of the screened Coulomb potential is for the long time estimate.      
\end{remark}
We can easily deduce short time propagation of position moments from propagation of velocity moments. 
\begin{corollary}\label{corollary short time}
Let the assumptions of Theorem \ref{3d estZn} hold and let $T_{\ast}$ be as in the conclusion of Theorem \ref{3d estZn}. Then, it holds that  
\begin{align*}
N_{n}(t)\leq (N_{n}^{\frac{1}{n}}(0)+ \overline{M}_{n}t)^{n}\ \mbox{for all}\ t\in[0,T_{\ast}],
\end{align*}
where $\overline{M}_{3}=2^{\frac{3}{2}}M_{3}(0)$ and $\overline{M}_{4}=2^{7}M_{4}(0)$. 
\begin{proof}
Integrating by parts and using  $\mathrm{div}_v(v \wedge B)=0$ we see that   \begin{align}
\frac{d}{dt}N_{n}(t)&=-\int_{\mathbb{R}^{6}}\left\vert x\right\vert^{n}\left(v\cdot\nabla_{x}f(t,x,v)-(v\wedge B(t,x)+\nabla_{x}K\star\rho_{f}(t,x))\cdot\nabla_{v}f(t,x,v)\right)\ dxdv \notag\\
&=-\int_{\mathbb{R}^{6}}\left\vert x\right\vert^{n}v\cdot \nabla_{x}f(t,x,v)\ dxdv
\notag\\&=n\int_{\mathbb{R}^{6}}\left\vert x\right\vert^{n-2}x\cdot v f(t,x,v)\ dxdv. \label{short time derivative of N}      
\end{align} 
Therefore it follows that 
\begin{align*}
\frac{d}{dt}N_{n}(t)&\leq n\int_{\mathbb{R}^{6}}\left\vert x\right\vert^{n-1}\left\vert v\right\vert f(t,x,v)\ dxdv\\
&\leq n\left(\int_{\mathbb{R}^{6}}\left\vert x\right\vert^{n}f(t,x,v)\ dxdv\right)^{\frac{n-1}{n}}\left(\int_{\mathbb{R}^{6}}\left\vert v\right\vert^{n}f(t,x,v)\ dxdv \right)^{\frac{1}{n}}\leq nN_{n}^{\frac{n-1}{n}}(t)M_{n}^{\frac{1}{n}}(t).        
\end{align*}
By Theorem \ref{3d estZn} it holds that 
\begin{align*}
M_{n}^{\frac{1}{n}}(t)\leq \overline{M}_{n} \ \mbox{for all} \ t\in [0,T_{\ast}].      
\end{align*}
Consequently we get 
\begin{align*}
\frac{d}{dt}N_{n}(t)\leq n\overline{M}_{n}N_{n}^{\frac{n-1}{n}}(t),    
\end{align*}
which implies 
\begin{align*}
\frac{d}{dt}N_{n}^{\frac{1}{n}}(t)\leq    \overline{M}_{n}. 
\end{align*}
Thus, we obtain 
\begin{align*}
N_{n}^{\frac{1}{n}}(t)-N_{n}^{\frac{1}{n}}(0)\leq \overline{M}_{n}t,     
\end{align*}
or equivalently  
\begin{align*}
N_{n}(t)\leq (N_{n}^{\frac{1}{n}}(0)+ \overline{M}_{n}t)^{n}.   \end{align*}
\end{proof}
\end{corollary}
\subsubsection{Long time propagation of moments of order $4$}\label{propagation long 4}
In variance with the 2D case, to establish propagation of moments globally, more care is needed. The key additional component is to first propagate the Eulerian moments which we recall are defined by 
\begin{align}
L_{n}(t)= \int_{\mathbb{R}^{6}}\left\vert x-tv\right\vert^{n}f(t,x,v)\ dxdv. \label{Eulerian moments}      
\end{align}
This approach is inspired by \cite{lafleche2021global}, which in turn draws from the methods of \cite{bardos1985global}
and \cite{castella1999propagation}. To achieve this necessitates imposing additional decay assumptions \eqref{eq:B} on the magnetic field $B$ which are redundant in the 2D case. These assumptions are responsible for creating the desired dispersion inequality. 
In addition it is important to note the potential corresponds to taking the screened Coulomb potential $K(x)=\pm \frac{e^{-\kappa\left\vert x\right\vert }}{\left\vert x\right\vert}$. 
\begin{theorem} 
Let $d=3$, $n=4$, $T>0$ and let $K(x)=\pm \frac{e^{-\kappa\left\vert x\right\vert }}{\left\vert x\right\vert}$. Let the assumption of Theorem \ref{main thm intro} hold. 
Then, there is some continuous in time function $\Phi(t)$ such that for all $t\in [0,T]$ it holds that 
\begin{align*}
L_{4}(t)\leq \Phi_{}(t).     
\end{align*}
\label{propagation of Eulerian moments}
\end{theorem}
Let us first explain how global propagation of $M_{4}(t)$ follows from the global propagation of $L_{4}(t)$ as stated in Theorem \ref{propagation of Eulerian moments} together with the short time estimate stated in Theorem \ref{3d estZn}.  
\begin{corollary}
Let the assumptions of Theorem \ref{propagation of Eulerian moments} is a continuous in time function $\Phi(t)=\Phi(t,\left\Vert f^{\mathrm{in}}\right\Vert_{\infty},\mathbf{N}^{\mathrm{in}}_{4},\mathbf{M}^{\mathrm{in}}_{4})$ such that 
\begin{align*}
N_{4}(t)\leq \Phi(t),\ M_{4}(t)\leq \Phi(t)\ \mbox{for all}\ t\in[0,T].   
\end{align*}
\label{conclusion of propagatiom of eulerian}
\end{corollary}
\begin{proof}
In the following calculation we always take $n=4$. We claim that $N_{n}(t)$ is globally propagated, from which global propagation of $M_{n}(t)$ would follow directly. We compute the time derivative of $N_{n}(t)$. Integrating by parts we see that    
\begin{align}
\frac{d}{dt}N_{n}(t)&=-\int_{\mathbb{R}^{6}}\left\vert x\right\vert^{n}\left(v\cdot\nabla_{x}f(t,x,v)-(v\wedge B(t,x)+\nabla_{x}K\star\rho_{f}(t,x))\cdot\nabla_{v}f(t,x,v)\right)\ dxdv \notag\\
&=-\int_{\mathbb{R}^{6}}\left\vert x\right\vert^{n}v\cdot \nabla_{x}f(t,x,v)\ dxdv
\notag\\&=n\int_{\mathbb{R}^{6}}\left\vert x\right\vert^{n-2}x\cdot v f(t,x,v)\ dxdv \notag\\ &=\frac{n}{t}\int_{\mathbb{R}^{6}}\left\vert x\right\vert^{n-2}x\cdot (tv-x)f(t,x,v)\ dxdv+\frac{n}{t}\int_{\mathbb{R}^{6}}\left\vert x\right\vert^{n}f(t,x,v)\ dxdv. \label{time derivative of N}      
\end{align}
By Corollary \ref{corollary short time} there is some $T_{\ast}>0$ such that
\begin{align}
 N_{n}(t)\leq (N_{n}^{\frac{1}{n}}(0)+ \overline{M}_{n}t)^{n}\ \mbox{for all $t\in [0,T_{\ast}]$}.  \label{Nn est}   
\end{align}
For all $t\geq T_{\ast}$, the first term in the right-hand side of  \eqref{time derivative of N} is bounded by  
\begin{align}
\frac{1}{t}\left((n-1)\int_{\mathbb{R}^{6}}\left\vert x\right\vert^{n}f
(t,x,v)\ dxdv+\int_{\mathbb{R}^{6}}\left\vert x-tv\right\vert^{n}f(t,x,v)\ dxdv \right)\leq \frac{n}{T_{\ast}}\left(N_{n}(t) +{L}_{n}(t)\right). \label{firstterm Nn}    
\end{align}
In addition, by Theorem \ref{propagation of Eulerian moments} there is some $\Phi(t)=\Phi(t,\left\Vert f^{\mathrm{in}}\right\Vert_{\infty},\mathbf{M}_{n}^{\mathrm{in}} )$ such that 
\begin{align*}
L_{n}(t)\leq \Phi(t).    
\end{align*}
Hence, it follows from \eqref{time derivative of N} and \eqref{firstterm Nn}  that for all $t\geq T_{\ast}$ it holds that  \begin{align*}
\frac{d}{dt}N_{n}(t)\leq \frac{2n}{T_{\ast}}\left(N_{n}(t)+L_{n}(t)\right)\leq \frac{2n}{T_{\ast}}(N_{n}(t)+\Phi(t)),\end{align*}
and hence 
\begin{align*}
N_{n}(t)\leq N_{n}(T_{\ast})+\frac{2n}{T_{\ast}}\int_{T_{\ast}}^{t}N_{n}(\tau)+\Phi(\tau)\ d\tau.    
\end{align*}
By Gr\"onwall's lemma it follows that for all $t\in [T_{\ast},T]$ it holds that  
\begin{align*}
N_{n}(t)\leq \Psi(t)\ \mbox{for some continuous in time}\  \Psi(t).
\end{align*}
Together with \eqref{Nn est}, this proves that $N_{n}(t)$ propagates on $[0,T]$. The estimate for $M_{n}(t)$ now follows by observing that for all $t\geq T_{\ast}$
\begin{align*}
M_{n}(t)\leq \frac{2^{n-1}}{t^{n}}\left(L_{n}(t)+N_{n}(t)\right)\leq \frac{2^{n-1}}{T_{\ast}^{n}}\left(L_{n}(t)+N_{n}(t)\right).     
\end{align*}
That $M_{n}(t)$ is propagated for $t\in [0,T_{\ast})$ follows again from Theorem \ref{3d estZn}.  
\end{proof}
\textit{Proof of Theorem \ref{propagation of Eulerian moments}}. In the following calculation we always take $n=4$. The places in which we use the specific numerical value of $n$ are indicated. \\
\textbf{Step 1}. (Calculation of $\frac{d}{dt}L_{n}(t)$). We compute  $\frac{d}{dt}L_{n}(t)$.  
We have 
\begin{align}
\frac{d}{dt}L_{n}(t)&=\frac{d}{dt}\int_{\mathbb{R}^{6}}\left\vert x-tv\right\vert^{n} f(t,x,v)\ dxdv \notag\\
&=\int_{\mathbb{R}^{6}}\left\vert x-tv\right\vert^{n}\partial_{t}f(t,x,v)\ dxdv \notag\\
&-n\int_{\mathbb{\mathbb{R}}^{6}}\left\vert x-tv\right\vert^{n-2} (x-tv)\cdot vf(t,x,v)\ dxdv \notag\\
&=-\int_{\mathbb{R}^{6}}\left\vert x-tv\right\vert^{n}v\cdot \nabla_{x}f(t,x,v)\ dxdv \notag\\
&+\int_{\mathbb{R}^{6}}\left\vert x-tv\right\vert^{n} (\nabla_x K\star \rho_{f}(t,x)+v\wedge B(t,x))\nabla_{v}f(t,x,v)\ dxdv \notag\\
&-n\int_{\mathbb{R}^{6}}\left\vert x-tv\right\vert^{n-2}(x-tv)\cdot vf(t,x,v)\ dxdv. \label{1st calculation of Ln}
\end{align}
Integration by parts shows that 
\begin{align}
\int_{\mathbb{R}^{6}}\left\vert x-tv\right\vert^{n} v\wedge B(t,x)\nabla_{v}f(t,x,v)\ dxdv=& 
-\int_{\mathbb{R}^{6}}\mathrm{div}_{v}\left(\left\vert x-tv\right\vert^{n} v\wedge B(t,x)\right)f(t,x,v)\ dxdv \notag\\
=&-\int_{\mathbb{R}^{6}}\left\vert x-tv\right\vert^{n} \mathrm{div}_{v}(v\wedge B(t,x))f(t,x,v)\ dxdv \notag \\
&+nt\int_{\mathbb{R}^{6}}\left\vert x-tv\right\vert^{n-2}(x-tv)v\wedge B(t,x)f(t,x,v)\ dxdv \notag\\
=&    nt\int_{\mathbb{R}^{6}}\left\vert x-tv\right\vert^{n-2}(x-tv)v\wedge B(t,x)f(t,x,v)\ dxdv,  \label{1}
\end{align}
where we used that $\mathrm{div}_{v}(v\wedge B)=0$ in the last equation. Note that this is the point where the magnetic field contributes a non-trivial term. 
Furthermore, we have  
\begin{align}
\int_{\mathbb{R}^{6}}&\left\vert x-tv\right\vert^{n}\nabla_x K\star \rho_{f}(t,x)\nabla_{v}f(t,x,v)\ dxdv \notag \\
&=nt\int_{\mathbb{R}^{6}}\left\vert x-tv\right\vert^{n-2}\nabla_x K\star \rho_{f}(t,x)\cdot(x-tv)f(t,x,v)\ dxdv.  {\label{2}}
\end{align}
Also, note that 
\begin{align}
-\int_{\mathbb{R}^{6}}\left\vert x-tv\right\vert^{n}v\cdot \nabla_{x}f(t,x,v)\ dxdv=n\int_{\mathbb{R}^{6}}\left\vert x-tv\right\vert^{n-2}(x-tv)\cdot vf(t,x,v)\ dxdv.  
\label{3}
\end{align}
Substituting \eqref{1}-\eqref{3} inside \eqref{1st calculation of Ln} we infer that
\begin{align*}
\frac{d}{dt}L_{n}(t)=&nt\int_{\mathbb{R}^{6}}\left\vert x-tv\right\vert^{n-2}\nabla_x K\star \rho_{f}(t,x)\cdot(x-tv)f(t,x,v)\ dxdv\\
&+nt\int_{\mathbb{R}^{6}}\left\vert x-tv\right\vert^{n-2}(x-tv)v\wedge B(t,x)f(t,x,v)\ dxdv\coloneqq \mathcal{L}_{1}+\mathcal{L}_{2}.   
\end{align*}
We proceed by estimating $\mathcal{L}_{1},\mathcal{L}_{2}$ separately.\\ 
\textbf{Step 2}. (Estimate on $\mathcal{L}_{2}$ ). Note that 
\begin{align*}
tv\wedge B(t,x)&=(tv-x)\wedge B(t,x)+x\wedge B(t,x)=(tv-x)\wedge B(t,x)+\mathbf{b}(t,x).
\end{align*}
So we can write  
\begin{align}
\mathcal{L}_{2}=&-n\int_{\mathbb{R}^{6}}\left\vert x-tv\right\vert^{n-2}(x-tv)\cdot(x-tv)\wedge B(t,x)f(t,x,v)\ dxdv \notag\\
&+n\int_{\mathbb{R}^{6}}\left\vert x-tv\right\vert^{n-2}\mathbf{b}(t,x)\cdot(x-tv)f(t,x,v)\ dxdv.      \label{L2 splittng}  
\end{align}
The first term in \eqref{L2 splittng} is bounded by 
\begin{align}
n\left\Vert B(t,\cdot)\right\Vert_{\infty}\int_{\mathbb{R}^{6}}\left\vert x-tv\right\vert^{n}f(t,x,v)\ dxdv=n\left\Vert B(t,\cdot)\right\Vert_{\infty}L_{n}(t). \label{first term L2}       
\end{align}
By means of H\"older's inequality the second term in \eqref{L2 splittng} is bounded by 
\begin{align}
&n\left\Vert \mathbf{b}(t,\cdot)\right\Vert_{\infty} \left(\int_{\mathbb{R}^{6}}\left\vert x-tv\right\vert^{n}f(t,x,v)\ dxdv\right)^{\frac{n-1}{n}}\left(\int_{\mathbb{R}^{6}}f(t,x,v)\ dxdv\right)^{\frac{1}{n}} \notag\\
&=n\left\Vert \mathbf{b}(t,\cdot)\right\Vert_{\infty}\left\Vert f^{\mathrm{in}}\right\Vert_{1}  L_{n}^{\frac{n-1}{n}}(t). \label{sec term L2}    
\end{align}
Thus, putting \eqref{first term L2} and \eqref{sec term L2} inside \eqref{L2 splittng} it follows that 
\begin{align}
\mathcal{L}_{2}\leq n\left\Vert B(t,\cdot)\right\Vert_{\infty}L_{n}(t)+n\left\Vert \mathbf{b}(t,\cdot)\right\Vert_{\infty}\left\Vert f^{\mathrm{in}}\right\Vert_{1}  L_{n}^{\frac{n-1}{n}}(t). \label{L2 est} \end{align}\\
\textbf{Step 3}. (Estimate on $\mathcal{L}_{1}$). 
We have 
\begin{align}
\mathcal{L}_{1}\leq& nt \left\Vert \int_{\mathbb{R}^{3}}\left\vert x-tv\right\vert^{n-1}f(t,x, v)\ d v\right\Vert_{L^p_{x}}\left\Vert \nabla_x K\star \rho_{f}(t,\cdot)\right\Vert_{p'} \notag \\
=&nt^{n}\left\Vert \int_{\mathbb{R}^{3}}\left\vert \frac{x}{t}-v\right\vert^{n-1}f(t,x, v)\ d v\right\Vert_{L^p_{x}}\left\Vert \nabla_x K\star \rho_{f}(t,\cdot)\right\Vert_{p'} \notag \\
=&nt^{n} \left\Vert \int_{\mathbb{R}^{3}}\left\vert  v\right\vert^{n-1}f(t,x, v+\frac{x}{t})\ d v\right\Vert_{L^{p}_{x}}\left\Vert \nabla_x K\star \rho_{f}(t,\cdot)\right\Vert_{p'}.  
\label{I2 INE}
\end{align}
Note that $K$ verifies $\nabla_x K\in L^{c}(\mathbb{R}^{3})$ for all $c\in [1,\frac{3}{2})$. Choosing $q$ such that 
\begin{align}
1+\frac{1}{p'}=\frac{1}{c}+\frac{1}{q}\, \mbox{or equivalently}\ \frac{1}{p'}+\frac{1}{q'}=\frac{1}{c} \label{conjugation relation}   
\end{align}
we have that 
\begin{align}
\left\Vert \nabla_x K\star \rho_{f}(t,\cdot)\right\Vert_{p'}\leq \left\Vert \nabla_x K\right\Vert_{c} \left\Vert \rho_{f}(t,\cdot)\right\Vert_{q}.
\label{weak Young inequality}
\end{align}
Since $n=4$ it follows that  $\frac{n+3}{n+1}=\frac{7}{5}<\frac{3}{2}$ and we may pick some $c\in (\frac{7}{5},\frac{3}{2})$. 
Note that 
$p_{n,n-1}=(n+3)'=\frac{n+3}{n+2}$ and $p_{n}=\left(\frac{n+3}{n}\right)'=\frac{n+3}{3}$ and so 
\begin{align*}
\frac{1}{p_{n,n-1}'}+\frac{1}{p_{n}'}=\frac{1}{n+3}+\frac{n}{n+3}=\frac{n+1}{n+3}=\frac{5}{7}\geq \frac{1}{c}.    
\end{align*}
Consequently we can choose $p\leq p_{n,n-1}$ and  $1<q\leq p_{n}$ such that \eqref{conjugation relation} holds. We have by interpolation
\begin{align}
\left\Vert \rho_{f}(t,\cdot)\right\Vert_{q}\leq \left\Vert \rho_{f}(t,\cdot)\right\Vert_{1}^{1-\frac{p_{n}'}{q'}} \left\Vert \rho_{f}(t,\cdot)\right\Vert_{p_{n}}^{\frac{p_{n}'}{q'}}=\left\Vert f^{\mathrm{in}}\right\Vert_{1}^{1-\frac{p'_{n}}{q'}}\left\Vert \rho_{f}(t,\cdot)\right\Vert_{p_{n}}^{\frac{p_{n}'}{q'}}  . \label{density bound} 
\end{align}
Therefore, inserting \eqref{weak Young inequality} and \eqref{density bound} inside \eqref{I2 INE}  we get (recall that $a=\frac{3}{c}-1>1$ )
\begin{align}
\mathcal{L}_{1}&\leq nt^{n}  \left\Vert \nabla_x K\right\Vert_{c}  \left\Vert \int_{\mathbb{R}^{3}}f(t,x, v+\frac{x}{t})\left\vert  v\right\vert^{n-1}\ d v\right\Vert_{L^{p}_{x}}\left\Vert f^{\mathrm{in}}\right\Vert_{1}^{1-\frac{p'_{n}}{q'}}\left\Vert \rho_{f}(t,\cdot)\right\Vert_{p_{n}}^{\frac{p_{n}'}{q'}} \notag\\
 &\leq C\left\Vert \nabla_x K\right\Vert_{c} t^{n}\left\Vert f(t,\cdot)\right\Vert_{\infty}^{\frac{1}{c}}\left\Vert f^{\mathrm{in}}\right\Vert_{1}^{1-\frac{p'_{n}}{q'}}\left(\int_{\mathbb{R}^{6}}f(t,x, v+\frac{x}{t})\left\vert  v\right\vert^{n}dxd v \right)^{1+\frac{a}{n}} \notag\\
 &= C\left\Vert \nabla_x K\right\Vert_{c} t^{n}\left\Vert f^{\mathrm{in}}\right\Vert_{\infty}^{\frac{1}{c}}\left\Vert f^{\mathrm{in}}\right\Vert_{1}^{1-\frac{p'_{n}}{q'}}\left(\int_{\mathbb{R}^{6}}f(t,x, v)\left\vert v-\frac{x}{t}\right\vert^{n}dxd v \right)^{1+\frac{a}{n}} \notag\\
 &=C\left\Vert \nabla_x K\right\Vert_{c} t^{-a}\left\Vert f^{\mathrm{in}}\right\Vert_{\infty}^{\frac{1}{c}}\left\Vert f^{\mathrm{in}}\right\Vert_{1}^{1-\frac{p'_{n}}{q'}}\left(\int_{\mathbb{R}^{6}}f(t,x, v)\left\vert tv-x\right\vert^{n}dxd v \right)^{1+\frac{a}{n}}
=\frac{\mathcal{C}L_{n}^{1+\frac{a}{n}}(t)}{t^{a}} \label{est L1}
\end{align}
where $\mathcal{C}=\mathcal{C}\left(c,\left\Vert f^{\mathrm{in}}\right\Vert_{\infty},\left\Vert f^{\mathrm{in}}\right\Vert_{1} \right)\coloneqq C\left\Vert \nabla_x K\right\Vert_{c} \left\Vert f^{\mathrm{in}} \right\Vert_{\infty}^{\frac{1}{c}}\left\Vert f^{\mathrm{in}}\right\Vert_{1}^{1-\frac{p'_{n}}{q'}}$.\\
\textbf{Step 4}.   (Conclusion) Gathering inequalities \eqref{L2 est} and \eqref{est L1} yields 
\begin{align*}
\frac{d}{dt}L_{n}(t)\leq    n\left\Vert B(t,\cdot)\right\Vert_{\infty}L_{n}(t)+n\left\Vert \mathbf{b}(t,\cdot)\right\Vert_{\infty}\left\Vert f^{\mathrm{in}}\right\Vert_{1}L_{n}^{\frac{n-1}{n}}(t)+\frac{\mathcal{C}L_{n}^{1+\frac{a}{n}}(t)}{t^{a}}.  \end{align*}
Putting $\ell_{n}(t)\coloneqq 1+L_{n}(t)$ the latter inequality entails 
\begin{align*}
\frac{d}{dt}\ell_{n}(t)\leq n\left\Vert B(t,\cdot)\right\Vert_{\infty}\ell_{n}(t)+n\left\Vert \mathbf{b}(t,\cdot)\right\Vert_{\infty}\left\Vert f^{\mathrm{in}}\right\Vert_{1}\ell_{n}^{\frac{n-1}{n}}(t)+\frac{\mathcal{C}\ell_{n}^{1+\frac{a}{n}}(t)}{t^{a}}.    
\end{align*}
By hypothesis \eqref{eq:B} we have 
\begin{align*}
\left\Vert B(t,\cdot)\right\Vert_{\infty}\leq \frac{B_{0}}{t^{a}}     
\end{align*}
and therefore the first term is bounded by  
\begin{align*}
\frac{B_{0}n}{t^{a}}\ell_{n}(t)\leq \frac{B_{0}n}{t^{a}}\ell_{n}^{1+\frac{a}{n}}(t).     
\end{align*}
By hypothesis \eqref{eq:B} we also have 
\begin{align*}
\left\Vert \mathbf{b}(t,\cdot)\right\Vert_{\infty}\leq \frac{B_{0}}{t^{a}}     
\end{align*}
and therefore the second term is bounded by 
\begin{align*}
\frac{B_{0}n\left\Vert f^{\mathrm{in}}\right\Vert_{1}}{t^{a}}\ell_{n}^{\frac{n-1}{n}}(t)\leq \frac{B_{0}n\left\Vert f^{\mathrm{in}}\right\Vert_{1} \ell_{n}^{1+\frac{a}{n}}(t)}{t^{a}}.     
\end{align*}
Thus we deduce the inequality
\begin{align*}
\frac{d}{dt}\ell_{n}(t)\leq \left(nB_{0}\left(1+\left\Vert f^{\mathrm{in}}\right\Vert_{1} \right)+\mathcal{C}\right)\frac{\ell_{n}^{1+\frac{a}{n}}(t)}{t^{a}}. 
\end{align*}
Dividing by $\ell_{n}^{1+\frac{a}{n}}(t)$ implies the inequality 
\begin{align*}
-\frac{d}{dt}\ell_{n}^{-\frac{a}{n}}(t)\leq  \frac{\mathcal{C}'}{t^{a}},
\end{align*}
where $ \mathcal{C}'\coloneqq \frac{a}{n}\left(nB_{0}(1+\left\Vert f^{\mathrm{in}}\right\Vert_{1})+\mathcal{C}\right)$. 
Integrating from $T_{\ast}$ to $t$ gives 
\begin{align*}
 \ell^{-\frac{a}{n}}_{n}(T_{\ast})-\ell^{-\frac{a}{n}}_{n}(t)\leq \mathcal{C}'\left(\frac{T_{\ast}^{-(a-1)}}{a-1}-\frac{t^{-(a-1)}}{a-1}\right).   
\end{align*}
Thus we conclude that 
\begin{align}
\ell^{-\frac{a}{n}}_{n}(t)\geq \ell^{-\frac{a}{n}}_{n}(T_{\ast})-\mathcal{C}'\left(\frac{T_{\ast}^{-(a-1)}}{a-1}-\frac{t^{-(a-1)}}{a-1}\right). \label{ln}    
\end{align}
So in order to get a bound for $\ell_{n}(t)$ we wish to ensure that 
\begin{align}
\ell_{n}^{-\frac{a}{n}}(T_{\ast})-\frac{\mathcal{C}'T_{\ast}^{-(a-1)}}{a-1}>0\ \mbox{or equivalently}\ T_{\ast}^{a+1}> \frac{\mathcal{C}'}{a-1}\ell_{n}^{\frac{a}{n}}(T_{\ast}). \label{wish to show}     
\end{align}
Note that we have 
\begin{align}
\ell_{n}^{\frac{a}{n}}(T_{\ast})&\leq \left(1+2^{n-1}T_{\ast}^{n}\int_{\mathbb{R}^{6}}\left\vert v\right\vert^{n}f(T_{\ast},x,v)\ dxdv+2^{n-1}\int_{\mathbb{R}^{6}}\left\vert x\right\vert^{n}f(T_{\ast},x,v)\ dxdv  \right)^{\frac{a}{n}} \notag\\
&\leq 2^{\frac{a(n-1)}{n}}\left(1+T_{\ast}^{a}M_{n}^{\frac{a}{n}}(T_{\ast})+N_{n}^{\frac{a}{n}}(T_{\ast})  \right). \label{estonLn}   
\end{align}
In view of \eqref{estonLn}, to fulfil \eqref{wish to show} it suffices to demand: 
\begin{itemize}
    \item [](a) \begin{align*}
T_{\ast}^{a+1}\geq \frac{2^{\frac{a(n-1)}{n}}3\mathcal{C}'}{a-1}T_{\ast}^{a}M_{n}^{\frac{a}{n}}(T_{\ast})     
\end{align*}
\item [] (b) 
\begin{align*}
T_{\ast}^{a+1}\geq \frac{2^{\frac{a(n-1)}{n}}3\mathcal{C}'}{a-1}N_{n}^{\frac{a}{n}}(T_{\ast}) \end{align*}
\item [] (c) 
\begin{align*}
 T_{\ast}^{a+1}\geq \frac{2^{\frac{a(n-1)}{n}}3\mathcal{C}'}{a-1}.    
\end{align*} 
\end{itemize}
By Theorem \ref{3d estZn} inequality (a) will be implied from 
\begin{align}
T_{\ast}\geq \frac{2^{\frac{7a(n-1)}{n}}3\mathcal{C}'}{a-1}M_{4}^\frac{a}{4}(0).  
\label{ineq1}
\end{align}
By Corollary \ref{corollary short time} inequality (b) will be implied by  
\begin{align}
T_{\ast}^{a+1}\geq \left(\frac{2^{\frac{a(n-1)}{n}}3\mathcal{C}'}{a-1}\right)(N_{4}^{\frac{1}{4}}(0)+2^{7}M_{4}(0)T_{\ast})^{a}. \label{ineq2} 
\end{align}
As for inequality (c), it will be implied by 
\begin{align}
T_{\ast}\geq \left(\frac{2^{\frac{a(n-1)}{n}}3\mathcal{C}'}{a-1}\right)^{\frac{1}{a+1}}. 
\label{ineq3}
\end{align}
Using the explicit formula for $T_{\ast}$ we see that  \eqref{ineq1}, \eqref{ineq2} and \eqref{ineq3} are exactly  the smallness assumption \eqref{smallness assumption}. 
\qed

\subsubsection{Long time estimate for moments of arbitrary order} We can now conclude the proof of Theorem \ref{main thm intro}ii., by verifying global propagation moments of order $\geq 5$. We first prove the result for $n=5$ and then induct on $n$. We have that 
\begin{align*}
\frac{d}{dt}M_{n}(t)=\frac{d}{dt}\int_{\mathbb{R}^{6}}\left\vert v\right\vert^{n} f(t,x,v)\ dxdv&=-n\int_{\mathbb{R}^{6}}\left\vert v\right\vert^{n-2}v\cdot \nabla_{x}K\star \rho_{f}(t,x)f(t,x,v)\ dxdv\\
&\leq n\int_{\mathbb{R}^{6}}\left\vert v\right\vert^{n-1}\left\vert \nabla_{x}K\star \rho_{f}\right\vert(t,x)f(t,x,v)\ dxdv.     
\end{align*}
For $n=5$ we obtain 
\begin{align}
\frac{d}{dt}M_{5}(t)\leq 5\int_{\mathbb{R}^{6}}\left\vert v\right\vert^{4}\left\vert \nabla_{x} K\star \rho_{f}(t,x)\right\vert f(t,x,v)\ dxdv\leq 5\left\Vert \rho_{4}(t,\cdot)\right\Vert_{\alpha}\left\Vert \nabla_{x} K\star \rho_{f}(t,\cdot)\right\Vert_{\alpha'}. \label{ine}       
\end{align}
Now if we choose $\alpha'=p'_{5,4}=5(1+\frac{3}{5})=8,\theta_{5,4}=\frac{1}{p'_{5,4}}=\frac{1}{8}$ then by Lemma \ref{Magnetic-kinetic-interpolation} we get 
\begin{align*}
\left\Vert \rho_{4}(t,\cdot)\right\Vert_{\alpha} =\left\Vert \rho_{4}(t,\cdot)\right\Vert_{\frac{8}{7}}\leq \frac{8}{7}\left(\int_{\mathbb{R}^{6}}\left\vert v\right\vert^{5}f(t,x,v)\ dxdv \right)^{\frac{7}{8}}\left\Vert f^{\mathrm{in}}\right\Vert_{\infty}^{\frac{1}{8}}=\frac{8}{7}M_{5}^{\frac{7}{8}}(t)\left\Vert f^{\mathrm{in}}\right\Vert_{\infty}^{\frac{1}{8}}  .      
\end{align*}
For this choice we also have 
\begin{align*}
\left\Vert \nabla_x K\star\rho_{f}(t,\cdot)\right\Vert_{\alpha'}=\left\Vert \nabla_x K\star \rho_{f}(t,\cdot)\right\Vert_{8}\leq \left\Vert \nabla_x K\right\Vert_{\frac{3}{2},\infty}\left\Vert \rho_{f}(t,\cdot)\right\Vert_{\frac{21}{10}}.        
\end{align*}
By Lemma \ref{Magnetic-kinetic-interpolation} there holds the inequality 
\begin{align*}
\left\Vert \rho_{f}(t,\cdot)\right\Vert_{\frac{21}{10}}\leq \frac{4}{3}M_{\frac{33}{10}}^{\frac{10}{21}}(t)\left\Vert f^{\mathrm{in}}\right\Vert_{\infty}^{\frac{11}{21}}.     \end{align*}
Therefore, thanks to Corollary  \ref{conclusion of propagatiom of eulerian} we deduce that 
\begin{align*}
\left\Vert \rho_{f}(t,\cdot)\right\Vert_{\frac{21}{10}}\leq \Phi(t)     
\end{align*}
for some $\Phi(t)=\Phi(t)=\Phi(t,\left\Vert f^{\mathrm{in}}\right\Vert_{\infty},\mathbf{N}^{\mathrm{in}}_{4},\mathbf{M}^{\mathrm{in}}_{4})$. Putting the last inequality inside \eqref{ine} we obtain \begin{align*}
\frac{d}{dt}M_{5}(t)\leq \Phi(t)M_{5}^{\frac{7}{8}}(t) \ \mbox{for all}\  t\in [0,T].    \end{align*}
Solving the above inequality shows that 
\begin{align*}
M_{5}(t)\leq \widetilde{\Phi}\left(t,\left\Vert f^{\mathrm{in}}\right\Vert_{\infty},\mathbf{M}_{5}^{\mathrm{in}},\mathbf{N}^{\mathrm{in}}_{5}\right),\end{align*}
for some continuous in time function $\widetilde{\Phi}$. 
We proceed by induction on $n$. We have 
\begin{align}
 \frac{d}{dt}M_{n}(t)\leq n\left\Vert \rho_{n-1}(t,\cdot)\right\Vert_{\alpha_{n}}\left\Vert \nabla_x K\star \rho_{f}(t,\cdot)\right\Vert_{\alpha_{n}'}. \label{prefinalMn}      
\end{align}
Now we choose $\alpha_{n}'=p'_{n,n-1}=n+3, \theta_{n,n-1}=\frac{1}{n+3}$ so that 
\begin{align}
\left\Vert \rho_{n-1}(t,\cdot)\right\Vert_{\alpha_{n}}\leq \frac{n+3}{n+2}\left(\int_{\mathbb{R}^{6}}\left\vert v\right\vert^{n}f(t,x,v)\ dxdv\right)^{\frac{n+2}{n+3}}\left\Vert f^{\mathrm{in}}\right\Vert_{\infty}^{\frac{1}{n+3}}=\frac{n+3}{n+2}M_{n}^{\frac{n+2}{n+3}}(t)\left\Vert f^{\mathrm{in}}\right\Vert_{\infty}^{\frac{1}{n+3}}. \label{rhon-1 est}        
\end{align}
In addition, we have 
\begin{align*}
\left\Vert \nabla_x K\star \rho_{f}(t,\cdot)\right\Vert_{\alpha_{n}'}=\left\Vert \nabla_x K\star \rho_{f}(t,\cdot)\right\Vert_{n+3}\leq \left\Vert \nabla_x K\right\Vert_{\frac{3}{2},\infty}\left\Vert \rho_{f}(t,\cdot)\right\Vert_{\frac{3(n+3)}{n+6}}.         
\end{align*}
We search for $k$ which will make Lemma \ref{Magnetic-kinetic-interpolation} applicable, i.e., $k$ such that 
\begin{align*}
\frac{3n+9}{2n+3}=1+\frac{3}{k} \ \mbox{or equivalently }  k=\frac{6n+9}{n+6}.    
\end{align*}
It follows that 
\begin{align}
 \left\Vert \rho_{f}(t,\cdot)\right\Vert_{\frac{3(n+3)}{n+6}}\leq \frac{4}{3}\left(\int_{\mathbb{R}^{6}}\left\vert v\right\vert^{k}f(t,x,v)\ dxdv \right)^{1-\theta_{k}}\left\Vert f^{\mathrm{in}}\right\Vert_{\infty}^{\theta_{n}}= \frac{4}{3}M_{k}^{1-\theta_{k}}(t)\left\Vert f^{\mathrm{in}}\right\Vert_{\infty}^{\theta_{n}}.  \label{estk}   \end{align}
Substituting \eqref{rhon-1 est} and \eqref{estk} inside \eqref{prefinalMn} we obtain 
\begin{align*}
\frac{d}{dt}M_{n}(t)\leq  \frac{4n(n+3)}{3(n+2)}M_{k}^{1-\theta_{k}}(t)M_{n}^{\frac{n+2}{n+3}}(t)\left\Vert f^{\mathrm{in}}\right\Vert_{\infty}^{\frac{1}{n+3}+\theta_{n}}.    
\end{align*}
Since $k<n-1$ whenever 
$n\geq 5$ and since $\frac{n+2}{n+3}<1$, the claim now follows by induction hypothesis and the last inequality. To show that $N_{n}(t)$ propagates globally we apply exactly the same argument demonstrated in Corollary \eqref{corollary short time}. 

\section{Propagation of regularity} \label{sec prop of regularity}
In this section we apply the propagation of moments in order to conclude propagation of regularity. The following corollary gives higher integrability of the density $\rho_{f}(t,\cdot)$, which will be subsequently used to prove boundedness of $\rho_{f}(t,\cdot)$. Recall that $ b $ is given by \eqref{definition of frakb} and that $ b '=\frac{ b }{ b -1}$. 
\begin{corollary}
\label{higher integrability of density}
Let the assumptions of Theorem \ref{main thm intro} hold and assume that $n>d( b '-1)$ where $\frac{1}{ b }+\frac{1}{ b '}=1$. Then $\rho_{f}\in L^{\infty}([0,T];L^{ b '}(\mathbb{R}^{d}))$. 
\end{corollary}
\begin{proof}
By Lemma \ref{Magnetic-kinetic-interpolation} we have 
\begin{align}
\left\Vert \rho_{f}(t,\cdot)\right\Vert_{p_{n}}\leq CM_{n}^{1-\theta_{n}}(t)\left\Vert f^{\mathrm{in}}\right\Vert_{\infty}^{\theta_{n}} \label{lemma reminder}
\end{align}
Since $n\geq d( b '-1)$ we see that $p_{n}\geq  b '$. 
Choosing $\lambda_{n}\in [0,1]$ such that 
\begin{align*}
\frac{1}{ b '}=\frac{\lambda_{n}}{p_{n}}+\frac{1-\lambda_{n}}{1}    
\end{align*}
we conclude by interpolation and \eqref{lemma reminder} that 
\begin{align*}
\left\Vert \rho_{f}(t,\cdot)\right\Vert_{ b '}\leq \left(CM_{n}^{(1-\theta_{n})}(t)\left\Vert f^{\mathrm{in}}\right\Vert_{\infty}^{\theta_{n}}\right)^{\lambda_{n}}\left\Vert f^{\mathrm{in}}\right\Vert_{1}^{1-\lambda_{n}}  .     
\end{align*}
The latter inequality together with Theorem \ref{main thm intro} imply the announced result.    
\end{proof}
\begin{theorem}
\label{thm_bound_rho}
Let the assumptions of Corollary \ref{higher integrability of density} hold. Then, it holds that $\rho_{f}\in L^{\infty}([0,T];L^{\infty}(\mathbb{R}^{d}))$ with the estimate 
\begin{align*}
\left\Vert \rho_{f}(t,\cdot)\right\Vert_{\infty}\leq  
c_{n}\exp\left(n\left\Vert \nabla_x K\right\Vert_{ b }\left\Vert \rho_{f}\right\Vert_{L^{\infty}_{t}L^{ b '}_{x}}t\right)\left\Vert (1+\left\vert  v\right\vert^{n}) f^{\mathrm{in}}(x, v)\right\Vert_{L^\infty_{x, v}}\mathbb{
}\ \mbox{for all}\ t\in[0,T]. \end{align*}
Consequently, there is a constant $\phi_{0}>0$ such that for all $p\in [1,\infty]$ it holds that 
\begin{align*}
\left\Vert \rho_{f}\right\Vert_{L^{\infty}_{t}L^{p}_{x}}\leq \phi_{0}.
\end{align*}
\end{theorem}
\begin{proof}
The main component of the proof is based on showing the propagation of the velocity weighted $L^
{\infty}$ norm  $\left\Vert (1+\left\vert  v\right\vert^{n} )f(t,x, v)\right\Vert_{L^\infty_{x, v}} $. We first propagate the velocity weighted $L^{p}$ norms for finite $p$. Given $1\leq p<\infty$ we compute that  
\begin{align*}
\frac{d}{dt}\int_{\mathbb{R}^{2d}}\left\vert (1+\left\vert  v\right\vert^{n} )f(t,x, v)\right\vert^{p}\ dxd v=&p\int_{\mathbb{R}^{2d}}(1+\left\vert  v\right\vert^{n})^{p-1}f^{p-1}(t,x, v)\partial_{t}f(t,x, v)\ dxd v\\
=&-p\int_{\mathbb{R}^{2d}}(1+\left\vert  v\right\vert^{n})^{p-1}f^{p-1}(t,x, v) v\cdot\nabla_{x}f(t,x, v)\ dxd v\\
&-p\int_{\mathbb{R}^{2d}}(1+\left\vert  v\right\vert^{n})^{p-1}f^{p-1}(t,x, v) v\wedge B(t,x)\cdot \nabla_{ v}f(t,x, v)\ dxd v\\
&-p\int_{\mathbb{R}^{2d}}(1+\left\vert  v\right\vert^{n})^{p-1}f^
{p-1}(t,x, v)\nabla_{x}K\star \rho_{f}(t,x)\cdot\nabla_{ v}f(t,x, v)\ dxd v. 
\end{align*}
First, we recognize that 
\begin{align}
-p&\int_{\mathbb{R}^{2d}}(1+\left\vert  v\right\vert^{n})^{p-1}f^{p-1}(t,x, v) v\cdot\nabla_{x}f(t,x, v)\ dxd v\\
&=-\int_{\mathbb{R}^{2d}}\nabla_{x} f^{p}(t,x, v)\cdot  v (1+\left\vert  v\right\vert^{n})\ dxd v=0, \label{vanishing}    
\end{align}
where the last equation is due to integration by parts.  
Furthermore, we have 
\begin{align*}
-\int_{\mathbb{R}^{2d}}&(1+\left\vert  v\right\vert^{n})^{p-1}\nabla_{ v}f^{p}(t,x, v)\nabla_{x}K\star \rho_{f}(t,x)\ dxd v\\
&=(p-1)\int_{\mathbb{R}^{2d}}(1+\left\vert v\right\vert^{n})^{p-2}\nabla_{ v}\left\vert  v\right\vert^{n}\cdot\nabla_{x}K\star\rho_{f}(t,x)f^{p}(t,x, v)\ dxd v\\
&=n(p-1)\int_{\mathbb{R}^{2d}}(1+\left\vert  v\right\vert^{n} )^{p-2}\left\vert  v\right\vert^{n-2} v\cdot \nabla_{x}K\star \rho_{f}(t,x)f^{p}(t,x, v)\ dxd v\\
&\leq n(p-1)\int_{\mathbb{R}^{2d}}(1+\left\vert  v\right\vert^{n})^{p-2}\left\vert  v\right\vert^{n-1}\left\vert \nabla_{x}K\star \rho_{f}\right\vert(t,x) f^{p}(t,x, v)\ dxd v.    
\end{align*}
Using the estimate 
\begin{align*}
\left\vert  v\right\vert^{n-1}\leq \frac{1}{n}+\frac{n-1}{n}\left\vert  v\right\vert^{n}\leq 1+\left\vert  v\right\vert^n      
\end{align*}
we obtain 
\begin{align}
&\textbf{}\left\vert \int_{\mathbb{R}^{2d}}(1+\left\vert  v\right\vert^{n})^{p-1}\nabla_{ v}f^{p}(t,x, v)\nabla_{x}K\star \rho_{f}(t,x)\ dxd v \right\vert \notag\\
&\leq n(p-1)\left\Vert \nabla_{x} K\star \rho_{f}(t,\cdot)\right\Vert_{\infty}\int_{\mathbb{R}^{2d}}(1+\left\vert  v\right\vert^{n})^{p-1}f^{p}(t,x, v)\ dxd v \notag\\
&\leq n(p-1)\left\Vert \nabla_{x} K\right\Vert_{ b ,\infty}\left\Vert \rho_{f}\right\Vert_{L^{\infty}_{t}L^{ b '}_{x}}\int_{\mathbb{R}^{2d}}(1+\left\vert  v\right\vert^{n})^{p}f^{p}(t,x, v)\ dxd v. \label{second ine}        
\end{align}
Note that by Corollary \ref{higher integrability of density} we have $\rho_{f}\in L^{\infty}([0,T];L^{ b '}(\mathbb{R}^{d}))$.  Observing that $\mathrm{div}_{ v}( v\wedge B)=0$ and integrating by parts we get 
\begin{align}
-&\int_{\mathbb{R}^{2d}}\nabla_{ v}f^{p}( v \wedge B)(1+\left\vert  v\right\vert^{n})^{p-1}\ dxd v \notag\\
&=\int_{\mathbb{R}^{2d}}\nabla_{ v}(1+\left\vert  v\right\vert^{n})^{p-1} v\wedge B(t,x)f^{p}(t,x, v)\ dxd v \notag\\
&+\int_{\mathbb{R}^{2d}}\mathrm{div}_{ v}( v\wedge B(t,x))(1+\left\vert  v\right\vert^{n})^{p-1}f^{p}(t,x, v)\ dxd v
\notag\\
&=n(p-1)\int_{\mathbb{R}^{2d}}(1
+\left\vert  v\right\vert^{n} )^{p-2}\left\vert  v\right\vert^{n-2} v\cdot v\wedge B(t,x) f^{p}(t,x, v)\ dxd v=0, \label{second vanishing}     
\end{align}
where in the last equation we used that $ v\cdot  v\wedge B=0$. Combining \eqref{vanishing},\eqref{second ine} and \eqref{second vanishing} we infer 
\begin{align*}
\frac{d}{dt}\int_{\mathbb{R}^{2d}}\left\vert (1+\left\vert  v\right\vert^{n} )f(t,x, v)\right\vert^{p}\ dxd v\leq n(p-1)\left\Vert \nabla_x K\right\Vert_{ b ,\infty}\left\Vert \rho_{f}\right\Vert_{L^{\infty}_{t}L^{ b '}_{x}}\int_{\mathbb{R}^{2d}}\left\vert (1+\left\vert  v\right\vert^n)f(t,x, v)\right\vert^{p}\ dxd v,         
\end{align*}
which by Gr\"onwall's lemma implies 
\begin{align*}
&\int_{\mathbb{R}^{2d}}\left\vert (1+\left\vert  v\right\vert^{n} )f(t,x, v)\right\vert^{p}\ dxd v\\
&\leq \exp\left(n(p-1)\left\Vert \nabla_x K\right\Vert_{ b ,\infty}\left\Vert \rho_{f}\right\Vert_{L^{\infty}_{t}L^{ b '}_{x}}t\right)\int_{\mathbb{R}^{2d}}\left\vert (1+\left\vert  v\right\vert^{n} )f^{\mathrm{in}}(x, v)\right\vert^{p}\ dxd v.   
\end{align*}
Taking the $p-1$ root and letting $p\rightarrow \infty$ yields 
\begin{align}
\left\Vert (1+\left\vert  v\right\vert^{n}) f(t,x, v)\right\Vert_{L^\infty_{x, v}}\leq \exp\left(n\left\Vert \nabla_x K\right\Vert_{ b ,\infty}\left\Vert \rho_{f}\right\Vert_{L^{\infty}_{t}L^{ b '}_{x}}t\right)\left\Vert (1+\left\vert  v\right\vert^{n}) f^{\mathrm{in}}(x, v)\right\Vert_{L^\infty_{x, v}}. \label{weighted Linfty prefinal}  
\end{align}
Thanks to \eqref{weighted Linfty prefinal} we get that  
\begin{align*}
\rho_{f}(t,x)=\int_{\mathbb{R}^{d}}\frac{1}{1+\left\vert  v\right\vert^{n}}(1+\left\vert  v\right\vert^{n} )f(t,x, v)\ d v&\leq \left\Vert (1+\left\vert  v\right\vert^{n})f(t,x, v)\right\Vert_{L^\infty_{x, v}}\int_{\mathbb{R}^{d}}\frac{1}{1+\left\vert  v\right\vert^{n}}\ d v\\
&\leq c_{n}\exp\left(n\left\Vert \nabla_x K\right\Vert_{ b ,\infty}\left\Vert \rho_{f}\right\Vert_{L^{\infty}_{t}L^{ b '}_{x}}t\right)\left\Vert (1+\left\vert  v\right\vert^{n}) f^{\mathrm{in}}(x, v)\right\Vert_{L^\infty_{x,v}}.  \end{align*}
Note that $n>d$ and therefore $c_{n}=\int_{\mathbb{R}^{d}}\frac{1}{1+\left\vert v\right\vert^{n}}\ dv<\infty$. 
\end{proof}
We proceed by propagation Sobolev regularity.  We will employ the following end point case of the Calderon-Zygmund inequality as well as a slightly more general form of Lemma \ref{Magnetic-kinetic-interpolation}.
\begin{proposition} 
\label{near boundedness}  
\textup{(\cite[Proposition 7.7]{bahouri2011fourier})}
Let $s>0$  and let $a\in[1,\infty)$ and $b\in[1,\infty]$. Then, there is some $C>0$ such that 
\begin{align*}
\left\Vert \nabla^2 V\star \mu \right\Vert_{\infty}\leq 
C\left(\mathrm{min} \left(\left\Vert\nabla V\star \mu\right\Vert_{b},\left\Vert \mu\right\Vert_{a}\right)+\left\Vert \mu\right\Vert_{\infty}(1+\left|\log\left\Vert \mu\right\Vert_{\infty}\right|\right)\log\left(\max\left(e,\left\Vert \mu\right\Vert_{\infty}\right)+\left\Vert \mu \right\Vert_{C^{0,s}}\right).
\end{align*} 
\end{proposition}
\begin{lemma}
\label{Magnetic-kinetic-interpolation_2} Suppose that $d\geq 1$, that $0\leq k\leq n$ and 
and put $p_{n,r,d} \coloneqq \frac{nr'+d}{n(r'-1)+d}, \theta_{n,r,d}\coloneqq\frac{r'}{p'_{n,r,d}}$. 
Then, for any $f\in L^{r}(\mathbb{R}^{d}\times\mathbb{R}^{d})$ with $M_{n,f} < \infty$
it holds that 
\begin{align*}
\left\Vert \int_{\mathbb{R}^{d}}\left| v\right|^{k}f(x, v)\ d v\right\Vert _{L^{p_{n,r,d}}_{x}}\leq C\left(\int_{\mathbb{R}^{2d}}\left| v\right|^{n}f(x, v)\ dxd v\right)^{1-\theta_{n,r,d}}\left\Vert f\right\Vert _{r}^{\theta_{n,r,d}},
\end{align*}
where $C=C_{n,r,d}$. 
\end{lemma}
\begin{proof} (\textit{of Theorem \ref{propagation of regularity intro}}). Put 
\begin{align*}
S(t)\coloneqq \left\Vert \nabla_{x}f(t,\cdot)\right\Vert_{L^{p}(e^{\lambda(t)w(x,v)})}^{p}+\left\Vert \nabla_{v}f(t,\cdot)\right\Vert_{L^{p}(e^{\lambda(t)w(x,v)})}^{p}.      
\end{align*}
We study the time derivative of $\frac{d}{dt}S(t)$ as follows.\\ 
\textbf{Step 1}. (\textit{Time derivative of}$\left\Vert \nabla_{x} f(t,\cdot)\right\Vert_{L^{p}(e^{\lambda(t)w(x, v)})}^{p}$). To make the equations lighter, we omit dependency on $(t,x, v)$ whenever there is no ambiguity. We compute that 
\begin{align*}
\frac{d}{dt}&\int_{\mathbb{R}^{2d}}e^{\lambda(t)w(x, v)}\left\vert \nabla_{x} f\right\vert^{p}(t,x, v)\ dxd v\\
=&p\int_{\mathbb{R}^{2d}}e^{\lambda(t)w(x, v)}\left\vert \nabla_{x}f\right\vert^{p-2}\nabla_{x}f\nabla_{x}\partial_{t}f\ dxd v+\int_{\mathbb{R}^{2d}}e^{\lambda(t)w(x, v)}\dot\lambda(t)w(x, v)\left\vert \nabla_{x} f\right\vert^{p}\ dxd v\\
=&-p\int_{\mathbb{R}^{2d}}e^{\lambda(t)w(x, v)}\left\vert \nabla_{x} f\right\vert^{p-2}\nabla_{x} f  \nabla_{x}\left( v\cdot \nabla_{x}f+( v\wedge B +\nabla_{x}K\star \rho_{f})\cdot \nabla_{ v}f\right)\ dxd v\\
&+\int_{\mathbb{R}^{2d}}e^{\lambda(t)w(x, v)}\dot\lambda(t)w(x, v)\left\vert \nabla_{x}f\right\vert^{p}\ dxd v\\
=&-p\int_{\mathbb{R}^{2d}}e^{\lambda(t)w(x, v)}\left\vert \nabla_{x}f \right\vert^{p-2}\nabla_{x}f\cdot\nabla^{2}_{x}f  v \ dxd v\\
&-p\int_{\mathbb{R}^{2d}}e^{\lambda(t)w(x, v)}\left\vert \nabla_{x}f\right\vert^{p-2}\nabla_{x}fD_{x}( v\wedge B)\nabla_{v}f \ dxd v\\
&-p\int_{\mathbb{R}^{2d}}e^{\lambda(t)w(x, v)}\left\vert \nabla_{x}f\right\vert^{p-2}\nabla_{x}f\cdot (D_{x}\nabla_{ v}f v \wedge B)\ dxd v\\
&-p\int_{\mathbb{R}^{2d}}e^{\lambda(t)w(x, v)}\left\vert \nabla_{x}f\right\vert^{p-2}\nabla_{x}f\cdot \nabla_{x}^{2}K\star \rho_{f}\nabla_{ v}f\ dxd v\\
&-p\int_{\mathbb{R}^{2d}}e^{\lambda(t)w(x, v)}\left\vert \nabla_{x}f\right\vert^{p-2}\nabla_{x}f\cdot (D_{x}\nabla_{ v}f\nabla_{x}K\star \rho_{f}) \ dxd v\\
&+\int_{\mathbb{R}^{2d}}e^{\lambda(t)w(x, v)}\dot\lambda(t)w(x, v)\left\vert \nabla_{x}f\right\vert^{p}\ dxd v\coloneqq \sum_{k=1}^{6}T_{k}.
\end{align*}
First, note that integration by parts reveals that 
\begin{align}
T_{1}=-\int_{\mathbb{R}^{2d}}\nabla_{x}\left\vert \nabla_{x}f \right\vert^{p}\cdot  v e^{\lambda(t)w(x, v)}\ dxd v&=\lambda(t)\int_{\mathbb{R}^{2d}}\left\vert \nabla_{x}f\right\vert^{p}x\cdot  v e^{\lambda(t)w(x, v)}\ dxd v \notag\\
&\leq \lambda(t)\int_{\mathbb{R}^{2d}}\left\vert \nabla_{x}f\right\vert^{p}w(x, v)e^{\lambda(t)w(x, v)}\ dxd v, \label{T1 est}  
\end{align}
and 
\begin{align}
T_{3}&=-\int_{\mathbb{R}^{2d}}e^{\lambda(t)w(x, v)}\nabla_{ v}\left\vert \nabla_{x}f\right\vert^{p}( v\wedge B)\ dxd v \notag\\
&=\int_{\mathbb{R}^{2d}}e^{\lambda(t)w(x, v)}\left\vert \nabla_{x}f\right\vert^{p}\mathrm{div}_{ v}( v\wedge B)\ dxd v+\int_{\mathbb{R}^{2d}}e^{\lambda(t)w(x, v)}\lambda(t)\nabla_{ v}w(x, v) v\wedge B\left\vert \nabla_{x}f\right\vert^{p} \ dxd v =0,    \label{T3 vanish} 
\end{align}
because $\mathrm{div}_{v}( v\wedge B)=0$ and $\nabla_{ v}w(x, v)\cdot  v\wedge B=0$. 
Furthermore, we estimate  
\begin{align*}
T_{2}&=-p\int_{\mathbb{R}^{2d}}e^{\lambda(t)w(x, v)}\left\vert \nabla_{x}f\right\vert^{p-2}\nabla_{x}fD_{x}( v\wedge B)\nabla_{ v}f\ dxd v\\
&\leq p\left\Vert D_{x}B\right\Vert_{\infty} \int_{\mathbb{R}^{2d}}e^{\lambda(t)w(x, v)}\left\vert \nabla_{x}f\right\vert^{p-1}\left\vert  v\right\vert\left\vert \nabla_{ v}f\right\vert\ dxd v\\
&\leq p\left\Vert D_{x}B\right\Vert_{\infty} \left(\frac{p-1}{p}\int_{\mathbb{R}^{2d}}e^{\lambda(t)w(x, v)}\left\vert  v\right\vert^{\frac{p}{p-1}}\left\vert \nabla_{x}f\right\vert^{p}\  dxd v+\frac{1}{p} \int_{\mathbb{R}^{2d}}e^{\lambda(t)w(x, v)}\left\vert \nabla_{ v}f\right\vert^{p}\ dxd v\right).  
\end{align*}
Note that by Young's inequality we have 
\begin{align*}
\left\vert  v\right\vert^{\frac{p}{p-1}}\leq \frac{p-2}{2(p-1)}+\frac{p}{2(p-1)}\left\vert  v\right\vert^{2} \leq 1+\left\vert  v\right\vert^{2}\leq w(x, v).      
\end{align*}
Substituting this inequality in the previous estimate yields 
\begin{align}
T_{2}\leq \left\Vert D_{x}B\right\Vert_{\infty}(p-1)\int_{\mathbb{R}^{2d}}e^{\lambda(t)w(x, v)}w(x, v)\left\vert\nabla_{x}f\right\vert^{p}\ dxd v \\+\left\Vert D_{x}B\right\Vert_{\infty}\int_{\mathbb{R}^{2d}}e^{\lambda(t)w(x, v)}\left\vert \nabla_{ v}f\right\vert^{p}\ dxd v. \label{T2 est}         
\end{align}
Next we estimate $T_{5}$.   
\begin{align}
T_{5}&=-\int_{\mathbb{R}^{2d}}e^{\lambda(t)w(x, v)}\nabla_{ v}\left\vert \nabla_{x}f\right\vert^{p}\nabla_{x}K\star \rho_{f}\ dxd v \notag\\
&=\int_{\mathbb{R}^{2d}}e^{\lambda(t)w(x, v)}\lambda(t)\nabla_{ v}w(x, v)\nabla_{x}K\star\rho_{f}\ \left\vert \nabla_{x}f\right\vert^{p} dxd v \notag \\
&\leq \left\Vert \nabla_{x}K\right\Vert_{ b ,\infty}\left\Vert \rho_{f}(t,\cdot)\right\Vert_{L^{\infty}L^{ b '}}\int_{\mathbb{R}^{2d}}\lambda(t)e^{\lambda(t)w(x, v)}w(x, v)\left\vert \nabla_{x}f\right\vert^{p}\ dxd v \notag\\
&\leq \left\Vert \nabla_x K\right\Vert_{ b ,\infty}\phi_{0}\int_{\mathbb{R}^{2d}}\lambda(t)e^{\lambda(t)w(x, v)}w(x, v)\left\vert \nabla_{x}f\right\vert^{p}\ dxd v. \label{T5 est}
\end{align}
Combining \eqref{T1 est},\eqref{T2 est} and \eqref{T5 est} together we conclude that 
\begin{align}
T_{1}+T_{2}+T_{5}+T_{6}\leq&  \lambda(t)(1+\left\Vert D_{x}B\right\Vert_{\infty}(p-1)+ \left\Vert \nabla_x K\right\Vert_{ b ,\infty}\phi_{0}) \int_{\mathbb{R}^{2d}}e^{\lambda(t)w(x, v)}w(x, v)\left\vert \nabla_{x}f\right\vert^{p} \ dxd v \notag\\
&+\dot\lambda(t)\int_{\mathbb{R}^{2d}}e^{\lambda(t)w(x, v)}w(x, v)\left\vert \nabla_{x}f\right\vert^{p}\ dxd v+\left\Vert D_{x}B\right\Vert_{\infty} \int_{\mathbb{R}^{2d}}e^{\lambda(t)w(x, v)}\left\vert \nabla_{ v}f\right\vert^{p}\ dxd v \notag\\
\leq& \left\Vert D_{x}B\right\Vert_{\infty} \int_{\mathbb{R}^
{2d}}e^{\lambda(t)w(x, v)}\left\vert \nabla_{ v}f\right\vert^{p}\ dxd v  \label{T1T2T5T6 est},      
\end{align}
where the last inequality is due to the assumption that  $$\dot\lambda(t)+\lambda(t)(1+\left\Vert D_{x}B\right\Vert_{\infty}(p-1)+ \left\Vert \nabla_x K\right\Vert_{ b ,\infty}\phi_{0}) \leq 0.$$ 
To estimate $T_{4}$, observe that  
\begin{align*}
T_{4}\leq& \left\Vert \nabla^{2}K\star \rho_{f}(t,\cdot)\right\Vert_{\infty}\left(\frac{p-1}{p}\int_{\mathbb{R}^{2d}}e^{\lambda(t)w(x, v)}\left\vert \nabla_{x}f\right\vert^{p}\ dxd v+\int_{\mathbb{R}^{2d}}e^{\lambda(t)w(x, v)}\left\vert \nabla_{ v}f\right\vert^{p}\ dxd v \right).      
\end{align*}
By Proposition \ref{near boundedness} we have 
\begin{align*}
\left\Vert \nabla^{2}K\star \rho_{f}(t,\cdot) \right\Vert_{\infty}\leq C'\left(1+\log(1+\left\Vert \rho_{f}(t,\cdot)\right\Vert_{W^{1,q}})\right)    
\end{align*}
where $C'=C'(q,\phi_{0})$. 
Therefore we get 
\begin{align}
T_{4}\leq C'\left(1+\log(1+\left\Vert \rho_{f}(t,\cdot)\right\Vert_{W^{1,q}})\right)\left(\frac{p-1}{p}\int_{\mathbb{R}^{2d}}e^{\lambda(t)w(x, v)}\left\vert \nabla_{x}f\right\vert^{p}\ dxd v+\int_{\mathbb{R}^{2d}}e^{\lambda(t)w(x, v)}\left\vert \nabla_{ v}f\right\vert^{p}\ dxd v \right). \label{T4 est}     
\end{align}
By Lemma \ref{Magnetic-kinetic-interpolation_2} applied for $q'=p'+\frac{d}{n}$ we have the inequality  
\begin{align*}
\left\Vert \nabla \rho_{f}(t,\cdot)\right\Vert_{q}&\leq C\left(\int_{\mathbb{R}^{2d}}\left\vert  v\right\vert^{n} \left\vert \nabla_{x}f\right\vert\ dxd v \right)^{1-\theta_{n,p,d}}\left\Vert \nabla_{x}f(t,\cdot)\right\Vert_{p}^{\theta_{n,p,d}}\\
&\leq C\left( (1-\theta_{n,p,d})\int_{\mathbb{R}^{2d}}\left\vert  v\right\vert^{n}\left\vert \nabla_{x}f\right\vert\ dxd v  +\theta_{n,p,d}\left\Vert \nabla_{x}f(t,\cdot)\right\Vert_{p}\right).   
\end{align*}
We can further bound 
\begin{align*}
\int_{\mathbb{R}^{2d}}\left\vert  v\right\vert^{n}\left\vert \nabla_{x}f\right\vert\ dxd v&=\int_{\mathbb{R}^{2d}}\left\vert  v\right\vert^{n} e^{-\frac{\lambda(t)w(x, v)}{p}}e^{\frac{\lambda(t)w(x, v)}{p}}\left\vert \nabla_{x}f\right\vert(t,x, v)\ dxd v\\
&\leq \frac{p-1}{p}\int_{\mathbb{R}^{2d}}\left\vert  v\right\vert^{\frac{np}{p-1}}e^{-\frac{\lambda(t)w(x, v)}{p-1}}\ dxd v+\frac{1}{p}\int_{\mathbb{R}^{2d}}e^{\lambda(t)w(x, v)}\left\vert \nabla_{x}f\right\vert^{p}\ dxd v.   
\end{align*}
Noticing that the first integral is absolutely convergent we can write 
\begin{align*}
\frac{p-1}{p}\int_{\mathbb{R}^{2d}}\left\vert  v\right\vert^{\frac{np}{p-1}}e^{-\frac{\lambda(t)w(x, v)}{p-1}}\ dxd v\lesssim 1    
\end{align*}
and thus 
\begin{align*}
\int_{\mathbb{R}^{2d}}\left\vert  v\right\vert^{n}\left\vert \nabla_{x}f\right\vert\ dxd v\lesssim 1+\left\Vert \nabla_{x}f(t,\cdot)\right\Vert^{p}_{L^{p}(e^{\lambda(t)w(x, v)})}.      
\end{align*}
Hence we have proved that 
\begin{align*}
\left\Vert \nabla_{x}\rho_{f}(t,\cdot)\right\Vert_{q}\lesssim 1+\left\Vert \nabla_{x}f(t,\cdot)\right\Vert_{L^{p}(e^{\lambda(t)w(x, v)})}^{p}.       
\end{align*}
To summarize, \eqref{T1T2T5T6 est} and \eqref{T4 est} entail 
\begin{align}
\frac{d}{dt}&\int_{\mathbb{R}^{2d}}e^{\lambda(t)w(x, v)}\left\vert\nabla_{x}f \right\vert^{p}\ dxd v \notag\\
&\lesssim \left(1+\log(1+\left\Vert \nabla_{x}f(t,\cdot)\right\Vert_{L
^
{p}(e^{\lambda(t)w(x, v)})}^{p})\right)\\
&\times \left(\frac{p-1}{p}\int_{\mathbb{R}^{2d}}e^{\lambda(t)w(x, v)}\left\vert \nabla_{x}f\right\vert^{p}\ dxd v+\int_{\mathbb{R}^{2d}}e^{\lambda(t)w(x, v)}\left\vert \nabla_{ v}f\right\vert^{p}\ dxd v \right).   \label{x est} \end{align}
\textbf{Step 2}. (\textit{Time derivative of $\left\Vert \nabla_{ v}f\right\Vert_{L^{p}(e^{\lambda(t)w(x, v)})}^{p} $}). We compute  
\begin{align*}
\frac{d}{dt}&\int_{\mathbb{R}^{2d}}e^{\lambda(t)w(x,v)}\left\vert \nabla_{ v}f\right\vert^{p}\ dxd v\\
&=-p\int_{\mathbb{R}^{2d}}e^{\lambda(t)w(x, v)}\left\vert \nabla_{ v}f\right\vert^{p-2}\nabla_{ v}f(\nabla_{x}f+D_{ v}\nabla_{x}f\cdot  v+D_{ v}( v\wedge B)\nabla_{ v}f+\nabla_{ v}^{2}f v\wedge B+\nabla_{x} K\star \rho_{f}\nabla^{2}_{ v}f)\ dxd v\\
&+\int_{\mathbb{R}^{2d}}e^{\lambda(t)w(x, v)}\dot\lambda(t)w(x, v)\left\vert \nabla_{ v}f\right\vert^{p} dxd v\\
=&-p\int_{\mathbb{R}^{2d}}e^{\lambda(t)w(x, v)}\left\vert \nabla_{ v}f\right\vert^{p-2}\nabla_{ v}f\nabla_{x}f\ dxd v-p\int_{\mathbb{R}^{2d}}e^{\lambda(t)w(x, v)}\left\vert \nabla_{ v}f\right\vert^{p-2}\nabla_{ v}fD_{ v}\nabla_{x}f\cdot  v \ dxd v\\
&-p\int_{\mathbb{R}^{2d}}e^{\lambda(t)w(x, v)}\left\vert \nabla_{ v}f\right\vert^{p-2}\nabla_{ v}fD_{ v}( v\wedge B)\nabla_{ v}f\ dxd v-p\int_{\mathbb{R}^{2d}}e^{\lambda(t)w(x, v)}\left\vert \nabla_{ v}f\right\vert^{p-2}\nabla_{ v}f\nabla^{2}_{ v}f v\wedge B\ dxd v\\
&-p\int_{\mathbb{R}^{2d}}e^{\lambda(t)w(x, v)}\left\vert \nabla_{ v}f\right\vert^{p-2}\nabla_{ v}f\nabla_{ v}^{2}f\nabla_{x}K\star\rho_{f}\ dxd v+\int_{\mathbb{R}^{2d}}e^{\lambda(t)w(x, v)}\dot\lambda(t)w(x, v)\left\vert \nabla_{ v}f\right\vert^{p} dxd v=\sum_{k=1}^{6}J_{k}.   \end{align*}
Clearly by Young's inequality we have 
\begin{align}
J_{1}\leq (p-1)\int_{\mathbb{R}^{2d}}e^{\lambda(t)w(x, v)}\left\vert \nabla_{ v}f\right\vert^{p} \ dxd v+\int_{\mathbb{R}^{2d}}e^{\lambda(t)w(x, v)}\left\vert \nabla_{x}f\right\vert^{p}\ dxd v. \label{est J1}      
\end{align}
\\
In addition, integrating by parts we see that 
\begin{align}
J_{2}&=-\int_{\mathbb{R}^{2d}}e^{\lambda(t)w(x, v)} v\nabla_{x}\left\vert \nabla_{ v}f\right\vert^{p}\ dxd v\\&=\int_{\mathbb{R}^{2d}}\lambda(t)e^{\lambda(t)w(x, v)}x\cdot  v\left\vert \nabla_{ v}f\right\vert^{p}\ dxd v 
\leq \int_{\mathbb{R}^{2d}}\lambda(t)e^{\lambda(t)w(x, v)}w(x, v)\left\vert \nabla_{ v}f\right\vert^{p}\ dxd v \label{J2ine}
\end{align}
and 
\begin{align}
J_{4}&=-\int_{\mathbb{R}^{2d}}e^{\lambda(t)w(x, v)}\nabla_{ v}\left\vert \nabla_{ v}f\right\vert^{p} v\wedge B\ dxd v \\
&=\int_{\mathbb{R}^{2d}}\mathrm{div}( v\wedge B)\left\vert \nabla_{ v}f\right\vert^{p}e^{\lambda(t)w(x,v)}\ dxd v\\
&+\int_{\mathbb{R}^{2d}}e^{\lambda(t)w(x, v)}\lambda(t)\nabla_{ v}w(x, v) v\wedge B \left\vert \nabla_{ v}f\right\vert^{p} \ dxd v=0.   
\label{J4vanish}    \end{align}
Furthermore, one has 
\begin{align}
J_{3}\leq p\left\Vert B\right\Vert_{\infty} \int_{\mathbb{R}^{2d}}e^
{\lambda(t)w(x, v)}\left\vert \nabla_{ v}f\right\vert^{p}\ dxd v.  \label{J3 est}  
\end{align}
As for $J_{5}$, integrating by parts we obtain  
\begin{align}
J_5&=\int_{\mathbb{R}^{2d}}\left\vert \nabla_{ v}f\right\vert^{p}e^{\lambda(t)w(x,v)}\lambda(t)\nabla_{ v}w(x, v)\nabla_{x}K\star \rho_{f}\ dxd v \\
&\leq \left\Vert \nabla_{x}K\star \rho_{f}(t,\cdot)\right\Vert_{\infty} \int_{\mathbb{R}^{2d}}e^{\lambda(t)w(x, v)}\left\vert \nabla_{ v}f\right\vert^{p}\lambda(t)w(x, v)\ dxd v \\
&\leq \left\Vert \nabla_x K\right\Vert_{ b ,\infty}\left\Vert \rho_{f}(t,\cdot)\right\Vert_{ b '}\int_{\mathbb{R}^{2d}}e^{\lambda(t)w(x, v)}\left\vert \nabla_{ v}f\right\vert^{p}\lambda(t)w(x, v)\ dxd v \\
&\leq \phi_{0}\left\Vert \nabla_x K\right\Vert_{ b ,\infty}\int_{\mathbb{R}^{2d}}e^{\lambda(t)w(x, v)}\left\vert \nabla_{ v}f\right\vert^{p}\lambda(t)w(x, v)\ dxd v.  \label{J5 est}    
\end{align}
Since by assumption  $(1+\phi_{0}\left\Vert \nabla_x K\right\Vert_{ b ,\infty}) \lambda(t)+\dot\lambda(t)\leq 0$ we infer from the last inequality and from \eqref{J2ine} that  
\begin{align}
 J_{2}+J_{5}+J_{6}\leq 0.
 \label{J2+J5+J6 nonpositive}
\end{align}
To conclude, from \eqref{est J1}-\eqref{J2+J5+J6 nonpositive} we see that we have proved that 
\begin{align}
 \frac{d}{dt}\left\Vert \nabla_{ v}f\right\Vert_{L^{p}(e^{\lambda(t)w(x, v)})}^{p}
 &\leq (p-1+\left\Vert B\right\Vert_{\infty}p)\int_{\mathbb{R}^{2d}}e^{\lambda(t)w(x, v)}\left\vert \nabla_{ v}f\right\vert^{p}\ dxd v \\
 & \quad +\int_{\mathbb{R}^{2d}}e^{\lambda(t)w(x, v)}\left\vert \nabla_{x}f\right\vert^{p}\ dxd v. \label{xi est}      
\end{align}
Gathering the inequalities \eqref{x est} and \eqref{xi est} yields 
\begin{align*}
\frac{d}{dt}S(t)\leq \overline{S}\left(\log(1+S(t)) \right)(1+S(t))\ \mbox{for some constant}\ \overline{S}>0.     
\end{align*}
Solving the above inequality gives 
\begin{align*}
S(t)\leq \overline{S}(1+S(0))^{e^{\overline{S}t}}\ \mbox{for all}\ t\in [0,T]
\end{align*}
as wanted. 
\end{proof}

\section{Improved stability estimate}
\label{sec_stability_estimate}

Magnetic fields cannot do work; their contribution to the Lorentz force is always orthogonal to the velocity. 
\begin{lemma}
\textup{(No‑work identity)}
\label{lem:nowork}
Let $(X, V)$ be a characteristic \eqref{trajectories of Vlasov mag} of \eqref{MAGNETIC Vlasov Intro}. For every $k>0$ and $0\le s\le t\le T$, we have
\[
\frac{d}{dt}|V(t)|^{k}=-k|V(t)|^{k-2} \nabla_x K \star \rho_{f}(t,X(t))\cdot V(t), \qquad |V(t)|\le |V(s)|+\int_{s}^{t}|\nabla_x K \star \rho_{f}(\tau,X(\tau))|\,d\tau.
\]
\end{lemma}
\begin{proof}
Differentiate $|V|^{k}$ and note that the term $V\cdot(V\wedge B)$ vanishes. 
Integrating the case $k=1$ yields the bound.
\end{proof}
Let $f_1,f_2\ge0$ be two solutions of the magnetized Vlasov-Poisson equation \eqref{MAGNETIC Vlasov Intro} with common mass~$1$. For $p\ge1$ and $\lambda>0$ we recall the $p$-kinetic Wasserstein distance introduced in \cite{iacobelli_new_2022};
\[
W^p_{\lambda,p}\left(\mu,\nu\right) \coloneqq \inf_{\pi\in\Pi(\mu,\nu)}\int_{(\XRd)^2} \lambda|x-y|^p+|v-w|^p \:d\pi.
\]
Setting $\lambda=1$ gives the usual $W_p(\mu,\nu)$.

Let $Z_i \coloneqq (X_i,V_i)$ be the characteristics \eqref{trajectories of Vlasov mag} of $f_i$, $i = 1,2$, starting from $(x,v)$.
Let $\pi_0$ be an optimal coupling for $W_{p}(f_1^{\mathrm{in}},f_2^{\mathrm{in}})$ and $\pi_t \coloneqq (\pi_0)_\#(Z_1(t), Z_2(t))$. Define
\begin{align*}
    D_p(t) &\coloneqq\int_{(\XRd)^2} \lambda(t)|X_1(t;x,v)-X_2(t;y,w)|^p+|V_1(t;x,v)-V_2(t;y,w)|^p\:d\pi_0(x,v,y,w) \\
        &= \int_{(\XRd)^2} \lambda(t)|x-y|^p+|v-w|^p\:d\pi_t(x,v,y,w)
\end{align*}
where we set $\lambda(t) \coloneqq \abs{\log D_p(t)}^{\tfrac{p}{2}}$. We recall that $D_p$ controls the $p$-Wasserstein distance between the two solutions;
\begin{equation}\label{ineq:D_p controls W_p^p}
    W_p^p\left(\rho_{f_1}, \rho_{f_2}\right) \le \frac{D_p}{\lambda}, \quad W_p^p\left(f_1, f_2\right) \le D_p.
\end{equation}
\begin{proof} (\textit{of Theorem \ref{thm:Iacobelli's W_p}}).
Differentiating $D_p$ with respect to time yields
\begin{align*}
    \frac{1}{p}\dot D_p = &\:\frac{1}{p} \int_{(\XRd)^2} \dot{\lambda}\abs{X_1 - X_2}^p \: d\pi_0 \\
        &+ \int_{(\XRd)^2} \lambda\abs{X_1 - X_2}^{p-2}(X_1 - X_2)\cdot(V_1 - V_2) \: d\pi_0 \\
        &+ \int_{(\XRd)^2} \abs{V_1 - V_2}^{p-2}(V_1 - V_2)\cdot(-\nabla_x K \star \rho_{f_1} + \nabla_x K \star \rho_{f_2}) \: d\pi_0 \\
        &+ \int_{(\XRd)^2} \abs{V_1 - V_2}^{p-2} (V_1 - V_2) \cdot \left[V_2 \wedge \left(B(X_2) - B(X_1)\right)\right] \: d\pi_0 \coloneqq I_1 + \cdots + I_4.
\end{align*}
$I_1$ is special in the sense that it is the only term containing a time derivative of $\lambda$, which, we recall, depends on $D_p$ itself; thus we will estimate it later. The bounds on $I_2$ and $I_3$ follow exactly the lines of the proof of \cite[Theorem 1.11]{iacobelli_stability_2024}. 

For $I_2$, H\"older's inequality with respect to the measure $\pi_0$ directly yields
\begin{equation}\label{ineq:I_2 bound}
    I_2(t) \le \lambda^{\frac{1}{p}}(t) D_p(t). 
\end{equation}

For $I_3$, one splits the difference of force fields as
\begin{multline*}
    \abs{\nabla_x K \star \rho_{f_1}(X_1) - \nabla_x K \star \rho_{f_2}(X_2)} \\ \le \abs{\nabla_x K \star \rho_{f_1}(X_1) -  \nabla_x K \star \rho_{f_1}(X_2)} + \abs{\nabla_x K \star \rho_{f_1}(X_2) - \nabla_x K \star \rho_{f_2}(X_2)}.
\end{multline*}
The first term uses the $\log$-Lipschitz regularity of the force field satisfying a (screened) Poisson equation with bounded density (straightforward modification of e.g. \cite[Lemma 3.1]{loeper2006uniqueness} and \cite[Lemma 3.2]{han2017quasineutral}).
\begin{lemma}\label{thm:Log-Lip estimate}
    There is a constant $C > 0$ such that 
    \begin{equation*}
        \norm{\nabla_x K \star \rho_{f_i}(t)}_{\infty} \le C\left(1 + \norm{\rho_{f_i}(t)}_{\infty}\right)
    \end{equation*}
    and for all $x, y \in \Rd$ with $\abs{x - y} < 1/e$, it holds
    \begin{equation*}\label{eq:Log-Lip electrif field estimate}
        \abs{\nabla_x K \star \rho_{f_i}(t, x) - \nabla_x K \star \rho_{f_i}(t, y)} \le C\left(1 + \norm{\rho_{f_i}(t)}_{\infty}\right)\abs{x - y}\log\left(\frac{4\sqrt{3}}{\abs{x - y}}\right). 
    \end{equation*}
\end{lemma}
\noindent The second term corresponds to the $L^p$-estimate on the difference of force fields for the case $1 < p < +\infty$.
\begin{proposition}\label{thm:L^p estimate}
    There is a constant $C_{\HW} > 0$ that only depends on $1 < p < +\infty$ such that
    \begin{equation*}\label{eq:L^p force field estimate}
         \norm{\nabla_x K \star \rho_{f_1}(t) -\nabla_x K \star \rho_{f_2}(t)}_{p} \le C_{\HW} \max \left\{ \norm{\rho_{f_1}(t)}_{\infty}, \norm{\rho_{f_2}(t)}_{\infty} \right\}^{\frac{1}{p'}}W_p\left(\rho_{f_1}(t), \rho_{f_2}(t)\right).
    \end{equation*}
\end{proposition}
\begin{proof}
    The case $d=2$ with the Coulomb potential is the content of \cite[Proposition 1.8]{iacobelli_stability_2024}. It remains to deal with the case $d=3$ corresponding to the Yukawa potential (or Coulomb when $\kappa = 0$); that is, the Green function for the screened Poisson equation $\pm(\Delta - \kappa^2) K \star \rho_f = \rho_f$. Let $\varphi \in C_c^\infty(\R^d)$ be a test function and set $1/p + 1/p' \coloneqq 1$. By duality of $L^p$-norm, for every $1 \le j \le 3$,
    \begin{equation*}
        \norm{\partial_{x_j} K \star \left(\rho_{f_1} - \rho_{f_2}\right) }_p = \sup_{\varphi \in C_c^\infty(\R^d) \,:\, \norm{\varphi}_{p'} \le 1 } \int_{\Rd} \left(\rho_{f_1} - \rho_{f_2}\right) \partial_{x_j} K \star \varphi \: dx.
    \end{equation*}
    We denote by $\phi_j \coloneqq \partial_{x_j} K \star \varphi$, and Calderon--Zygmund's inequality \cite[Theorem II.11.4]{Galdi2011} yields $\norm{\nabla \phi_j}_{p'} \le \widetilde{C_{\HW}} \norm{\varphi}_{p'}$ for some constant $\widetilde{C_{\HW}} > 0$ that only depends on $1 < p <+\infty$. Replacing the supremum over the larger set
    \begin{equation*}
        \left\{\phi \in \dot{W}^{1,p'}(\Rd) \, : \, \norm{\nabla\phi}_{p'}/\widetilde{C_{\HW}} \le 1\right\} \supset \left\{ \varphi \in C_c^\infty(\Rd) \, : \norm{\varphi}_{p'} \le 1 \right\},
    \end{equation*}
    we obtain, by definition of dual homogeneous Sobolev norm, 
    \begin{align*}
        \norm{\partial_{x_j} K \star \left(\rho_{f_1}(t) - \rho_{f_2}(t)\right) }_p &\le \widetilde{C_{\HW}} \sup_{\phi \in \dot{W}^{1,p'}(\Rd) \, : \, \norm{\nabla\phi}_{p'} \le 1} \int_{\Rd} \left(\rho_{f_2}(t) - \rho_{f_1}(t)\right) \nabla\phi \: dx \\
            &\coloneqq \widetilde{C_{\HW}} \norm{\rho_{f_2}(t) - \rho_{f_1}(t)}_{\dot{W}^{-1,p}}.
    \end{align*}
    We conclude following the proof of \cite[Proposition 1.8]{iacobelli_stability_2024}, which gives
    \begin{equation*}
        \norm{\rho_{f_2}(t) - \rho_{f_1}(t)}_{\dot{W}^{-1,p}} \le \max \left\{ \norm{\rho_{f_1}(t)}_{\infty}, \norm{\rho_{f_2}(t)}_{\infty} \right\}^{\frac{1}{p'}}W_p\left(\rho_{f_1}(t), \rho_{f_2}(t)\right).
    \end{equation*}
\end{proof}
\noindent Then combining both Lemma \ref{thm:Log-Lip estimate} and Proposition \ref{thm:L^p estimate} with \eqref{ineq:D_p controls W_p^p}, one obtains, as long as the regime
\begin{equation}\label{eq:regime 1 over e}
    D_p(t) \le 1/e \text{ on } [0, T) 
\end{equation}
prevails,
\begin{equation}\label{ineq:I_3 bound}
    I_3(t) \le C_{p} \widetilde A(t) \lambda^{-\frac{1}{p}}(t) D_p(t)\abs{\log D_p(t)}, 
\end{equation}
where
\begin{equation*}
    \widetilde A(t) \coloneqq \max\left\{1, \norm{\rho_{f_2}(t)}_\infty + \norm{\rho_{f_1}(t)}_\infty^{\frac{1}{p}} \max\left\{\norm{\rho_{f_1}(t)}_\infty, \norm{\rho_{f_2}(t)}_\infty\right\}^{\frac{1}{p'}} \right\}
\end{equation*}
and $C_{p}$ is a constant that depends only on $1 < p <+\infty$. In the case $p = 1$, the coupling argument found in \cite{junne_stability_2025} yields the same estimate as \eqref{ineq:I_3 bound} with some constant $C_{1}$.

For $I_4$, we use Lemma \ref{lem:nowork}, namely, the no-work identity, to obtain a better control of the isolated velocity flow $V_2$ compared to \cite{rege2025stability} and \cite{junne_stability_2025}:
\noindent Combining H\"older's inequality with the no-work identity, we obtain
\begin{equation*}
    I_4(t) \le D_p^{\frac{1}{p'}}(t) \left(I_{4,1}(t) + I_{4,2}(t)\right),
\end{equation*}
where
\begin{multline*}
    I_{4,1}(t) \le \left(\int_{(\XRd)^2} \abs{V_2(0)}^p \abs{B(X_1) - B(X_2)}^p \: d\pi_0\right)^{\frac{1}{p}}, \\
    I_{4,2}(t) \le \left(\int_{(\XRd)^2} \left(\int_0^t \abs{\nabla_x K \star \rho_{f_2}(X_2(\tau))} \: d\tau\right)^p\abs{B(X_1) - B(X_2)}^p \: d\pi_0\right)^{\frac{1}{p}}.
\end{multline*}
The first term is more subtle, so we will start with the latter. Thanks to the uniform bound on the force field we obtain
\begin{equation*}
    I_{4,2}(t) \le \left(\int_0^t C\left(1 + \norm{\rho_{f_2}(\tau)}_{\infty}\right) \: d\tau\right)\left(\int_{(\XRd)^2} \abs{B(X_1) - B(X_2)}^p \: d\pi_0\right)^{\frac{1}{p}}.
\end{equation*}
Now, since the external magnetic field belongs to $L^\infty([0, T); W^{1,\infty}(\X))$, it follows that $\abs{B(X_1) - B(X_2)} \le \widetilde C_B \abs{X_1 - X_2}$ for some constant $\widetilde C_B > 0$ and thus
\begin{equation}\label{ineq:I_4,2 bound}
    I_{4,2}(t) \le C_{B,2} \left(\int_0^t \widetilde A(\tau) \: d\tau\right) \lambda^{-\frac{1}{p}}(t) D_p^{\frac{1}{p}}(t), \quad C_{B,2} \coloneqq 2C\widetilde C_B.
\end{equation}

Regarding the first term $I_{4,1}$, note that $B$ is H\"older continuous for every exponent $\alpha \in (0, 1)$; $\abs{B(X_1) - B(X_2)} \le \overline{C_B}\abs{X_1 - X_2}^\alpha$. Combining this observation with H\"older's inequality for the exponents $1/\alpha + 1/(1/\alpha)' = 1$, we obtain
\begin{align*}
    I_{4,1}(t) &\le \overline{C_B}\left(\int_{(\XRd)^2} \abs{w}^p\abs{X_1 - X_2}^{\alpha p} \: d\pi_0(x,v,y,w)\right)^{\frac{1}{p}} \\
        &\le \overline{C_B}\left(\int_{\XRd} \abs{w}^{\left(\frac{1}{\alpha}\right)'p} \: df_2(0; y, w)\right)^{\frac{1}{(1/\alpha)'p}}\left(\frac{D_p}{\lambda}\right)^{\frac{\alpha}{p}}. 
\end{align*}
Of course, one wishes $\alpha$ to be the closest to $1$. We set, under the regime \eqref{eq:regime 1 over e},
\begin{equation*}
    \alpha \coloneqq 1 - \frac{1}{\abs{\log (D_p/\lambda)}}, \quad \Big(\left(1/\alpha\right)' = \abs{\log(D_p/\lambda)}\Big)
\end{equation*}
so that both $\alpha$ belongs to $(0, 1)$ and
\begin{equation*}
    \left(\frac{D_p}{\lambda}\right)^{\frac{\alpha}{p}} = e^{-1}\lambda^{-\frac{1}{p}}D_p^{\frac{1}{p}}.
\end{equation*}
Note that the choice $\lambda \coloneqq \abs{\log D_p}^{p/2}$ yields the following elementary inequality in the regime:
\begin{equation*}
    \abs{\log \left(\frac{D_p}{\lambda}\right)} \le \left(2 + \log\frac{p}{2}\right) \abs{\log D_p}.
\end{equation*}
Now, we recall the weighted Yudovich norm of a function $g$ with respect to a measure $\mu$ and a growth function $\Theta$;
\begin{equation*}
    \norm{g}_{Y^\Theta(d\mu)} \coloneqq \sup_{1 \le r < +\infty} \frac{\norm{g}_{L^r(d\mu)}}{\Theta(r)}.
\end{equation*}
It follows that
\begin{align*}
    \left(\int_{\XRd} \abs{w}^{\left(\frac{1}{\alpha}\right)'p} \: df_2^{\mathrm{in}}(y, w)\right)^{\frac{1}{(1/\alpha)'p}} \le \norm{\abs{w}}_{Y^\Theta(df_2^{\mathrm{in}}(y, w))} \Theta\left(p(1/\alpha)'\right).
\end{align*}
Combining these bounds, we get
\begin{equation}\label{ineq:I_4,1 bound}
    I_{4,1}(t) \le C_{B,1} \norm{\abs{w}}_{Y^\Theta(df_2^{\mathrm{in}}(y, w))} \lambda(t)^{-\frac{1}{p}} D_p(t)^{\frac{1}{p}} \Theta\left(\overline{C_p} \abs{\log D_p(t)}\right),
\end{equation}
where $C_{B,1} \coloneqq e^{-1}\overline{C_B}$ and $\overline{C_p} \coloneqq p(2 + \log(p/2))$.
Note that since $M_k(0) \le (C_0 k)^k$ for all $k > 0$ is assumed to hold for one of two solutions, say $f_2$, which corresponds to the growth function $\Theta(r) \coloneqq r$, then $\norm{\abs{w}}_{Y^\Theta(df_2^{\mathrm{in}}(y, w))} \le C_0$ and we obtain, combining the two estimates \eqref{ineq:I_4,1 bound} and \eqref{ineq:I_4,2 bound} respectively for $I_{4,1}(t)$ and $I_{4,2}(t)$, that
\begin{equation}\label{ineq:I_4 bound}
    I_4(t) \le \widetilde{C}_{p,B} \left(\int_0^t \widetilde A^p(\tau) \: d\tau\right)^{\frac{1}{p}} \lambda^{-\frac{1}{p}}(t) D_p(t) \abs{\log D_p(t)}, \quad \widetilde{C}_{p,B} \coloneqq \overline{C_p}C_0 C_{B,1} + C_{B,2}. 
\end{equation}

For $I_1$, note that if $\dot D_p \le 0$ there is nothing to do. Otherwise, the specific choice of $\lambda \coloneqq \abs{\log D_p}^{p/2}$ yields
\begin{equation*}
    I_1(t) = \frac{1}{2}\abs{\log D_p(t)}^{\frac{p}{2}-1} \left(\frac{-\dot D_p(t)}{D_p(t)}\right) \int_{(\XRd)^2} \abs{X_1 - X_2}^p \: d\pi_0 \le 0,
\end{equation*}
and thus, independently of the sign of $\dot D_p(t)$, in the regime \eqref{eq:regime 1 over e} it holds
\begin{align*}
    \frac{1}{p}\dot D_p(t) &\le I_2(t) + I_3(t) + I_4(t) \\
        &\le \lambda^{\frac{1}{p}}(t) D_p(t) + C_{p} \widetilde A(t) \lambda^{-\frac{1}{p}}(t) D_p(t)\abs{\log D_p(t)} + \widetilde{C}_{p,B} \left(\int_0^t \widetilde A(\tau) \: d\tau\right) \lambda^{-\frac{1}{p}}(t) D_p(t) \abs{\log D_p(t)} \\
        &\le \frac{2}{p}C_{p,B} A(t) D_p(t) \sqrt{\abs{\log D_p(t)}},
\end{align*}
where $C_{p,B} \coloneqq \tfrac{p}{2}(1 + C_{p} + \widetilde{C}_{p,B})$ and
\begin{equation}\label{eq:def A}
    A(t) \coloneqq \widetilde A(t) + \int_0^t \widetilde A(\tau) \: d\tau.
\end{equation}
Solving the differential inequality, we obtain
\begin{equation}\label{ineq:differential inequality D_p}
    D_p(t) \le \exp\left\{-\left(\sqrt{\abs{\log D_p(0)}} - C_{p,B}\int_0^t A(\tau) \: d\tau\right)^{2}\right\}.
\end{equation}
Note that the regime \eqref{eq:regime 1 over e} implies that \eqref{ineq:differential inequality D_p} holds if
\begin{equation}\label{ineq:initial condition regime}
    \sqrt{\abs{\log D_p(0)}} \ge C_{p,B}\int_0^t A(\tau) \: d\tau + 1.
\end{equation}
Now, recall that $D_p$ controls the $p$-Wasserstein distance between the two solutions; at time $t$ it holds $D_p(t) \le W_p^p\left(f_1(t), f_2(t)\right)$, while initially, by $W_p$-optimality of the coupling $\pi_0$,
\begin{equation*}
    D_p(0) \le \lambda(0) \int_{(\XRd)^2} \abs{x-y}^p + \abs{v-w}^p \: d\pi_0 = \abs{\log D_p(0)}^{\frac{p}{2}} W_p^p\left(f_1^{\mathrm{in}}, f_2^{\mathrm{in}}\right),
\end{equation*}
and, since near the origin the inverse of the function $s \mapsto s/\abs{\log s}^{\tfrac{p}{2}}$ behaves like $\tau \mapsto \tau\abs{\log \tau}^{\tfrac{p}{2}}$, there is a universal constant $c_0 > 0$ such that for sufficiently small initial Wasserstein distance $W_p^p(f_1^{\mathrm{in}}, f_2^{\mathrm{in}}) < c_0$ it holds
\begin{equation*}
    D_p(0) \le W_p^p\left(f_1^{\mathrm{in}}, f_2^{\mathrm{in}}\right) \sqrt{\abs{\log W_p^p\left(f_1^{\mathrm{in}}, f_2^{\mathrm{in}}\right)}}^p.
\end{equation*}
Thus \eqref{ineq:differential inequality D_p} implies
\begin{equation}\label{ineq:stability estimate}
        W_p^p\left(f_1(t), f_2(t)\right) \\ \le \exp\left\{-\left(\sqrt{\abs{\log\Bigg\{W_p^p(f_1^{\mathrm{in}}, f_2^{\mathrm{in}}\sqrt{\abs{\log W_p^p(f_1^{\mathrm{in}}, f_2^{\mathrm{in}})}}^p\Bigg\}}} \!- C_{p,B}\int_0^t A(\tau) \: d\tau\right)^2\right\}
    \end{equation}
as long as both $W_p^p(f_1^{\mathrm{in}}, f_2^{\mathrm{in}}) < c_0$ and \eqref{ineq:initial condition regime} holds; that is,
\begin{equation}\label{ineq:initial condition regime in terms of W_p^p}
    \sqrt{\abs{\log \left(W_p^p\left(f_1(0), f_2(0)\right) \sqrt{\abs{\log W_p^p\left(f_1(0), f_2(0)\right)}}^p\right)}} \ge C_{p,B}\int_0^t A(\tau) \: d\tau + 1.  
\end{equation}
\end{proof}
    
\begin{remark}
    The improvement coming from the no-work identity can be seen as follows: Recall that if one were to look at the evolution of the velocity itself and not its norm, then one obtains the following estimate (see \cite[Lemma 2.8]{rege2025stability}): for all $t \in [0, T), v \in \Rd$, it holds (say $\norm{B}_\infty < +\infty$)
    \begin{equation*}\label{ineq:estimate velocity field E B}
        \norm{V(t;\cdot,v)}_{\infty} \le \abs{v}e^{t\norm{B}_{\infty}} + \int_0^t \norm{\nabla_x K \star \rho_{f}(\tau)}_{\infty}e^{(t-\tau)\norm{B}_{\infty}} \: d\tau.
    \end{equation*}
    Now, this estimate would introduce, among other things, a prefactor $e^{t\norm{B}_{\infty}}$ in the estimate of $I_{4}(t)$ which becomes large when $t$ itself is; $\int_0^t \widetilde A(\tau) d\tau$ is replaced by $\int_0^t \widetilde A(\tau)e^{(t-\tau)\norm{B}_{\infty}}$ in \eqref{eq:def A}. This undermines the goal, namely, obtaining a large time interval $[0, T(\delta)]$ such that $W_p(f_1(t), f_2(t)) \lesssim 1$ given that $W_p(f_1^{\mathrm{in}}, f_2^{\mathrm{in}}) = \delta \ll 1$. Indeed, in light of \eqref{ineq:initial condition regime in terms of W_p^p}, then $T(\delta) \sim \log |\log \delta|$ without using the no-work identity, which corresponds to the time-interval originally obtained by Loeper \cite{loeper2006uniqueness} for the Vlasov-Poisson equation, while $T(\delta) \sim \sqrt{\abs{\log \delta}}$ with the no-work identity, coinciding with the improved control on time interval obtained using the kinetic Wasserstein distance for the Vlasov-Poisson equation found in \cite{iacobelli_new_2022} and \cite{iacobelli_stability_2024}.
\end{remark}

\section*{Acknowledgements}
We thank Hezekiah Grayer II for highlighting the no‑work identity. M. Iacobelli and A. Gagnebin acknowledge the support of the SNSF Starting grant  \emph{Challenges and Breakthroughs in the Mathematics of Plasmas},  TMSGI2\_226018. J. Junné acknowledges the support of the research program Open Competitie ENW (partly) financed by the Dutch Research Council (NWO) \href{https://www.nwo.nl/en/projects/ocenwm20251}{\includegraphics[height=\fontcharht\font`\B]{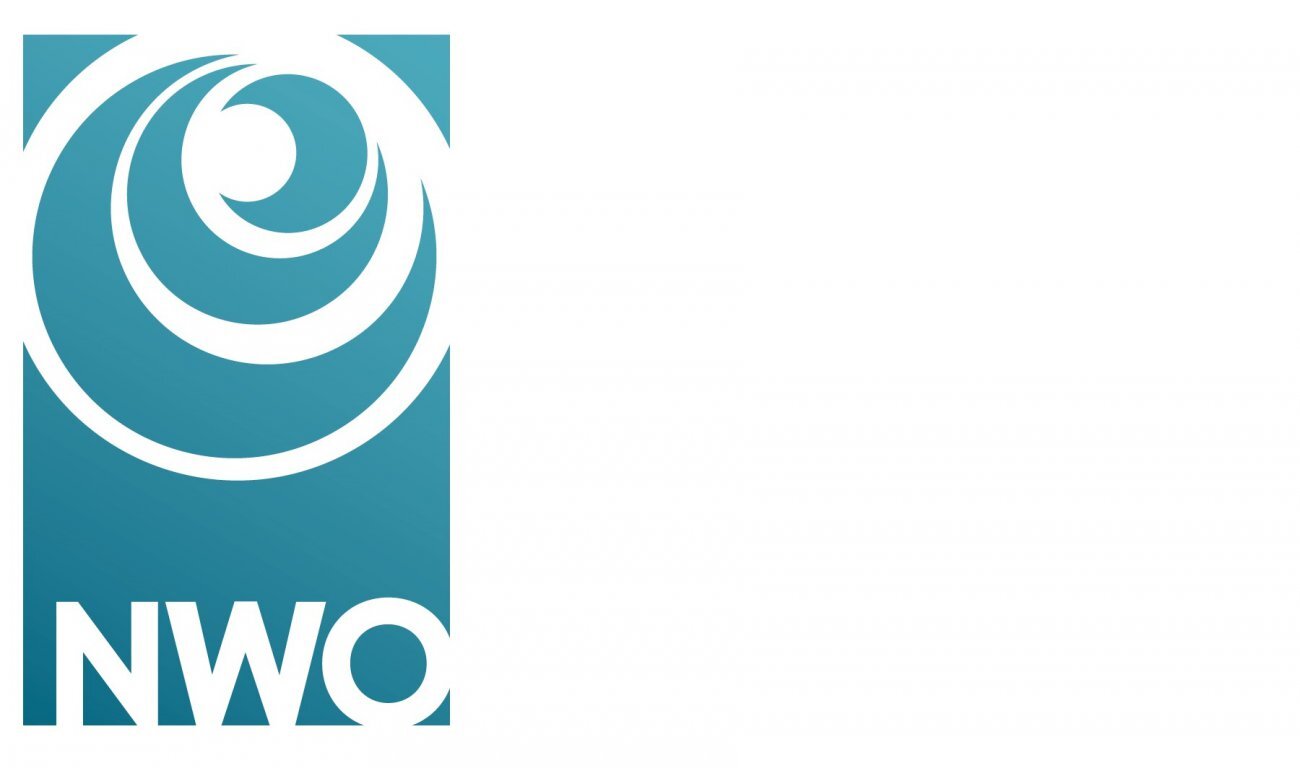}} for the project \emph{Interacting particle systems and Riemannian geometry} - OCENW.M20.251.


\bibliographystyle{abbrv}
\bibliography{biblio}

\begin{thebibliography}{10}

\bibitem{arsenev1975existence}
A.~A. Arsenev.
\newblock Existence in the large of a weak solution to the {V}lasov system of equations.
\newblock {\em Zhurnal Vychislitelnoi Matematiki i Matematicheskoi Fiziki}, 15:136--147, 1975.

\bibitem{bahouri2011fourier}
H.~Bahouri.
\newblock {\em Fourier analysis and nonlinear partial differential equations}.
\newblock Springer, 2011.

\bibitem{Bardos_Degond_1}
C.~Bardos and P.~Degond.
\newblock Existence globale des solutions des \'{e}quations de {V}lasov-{P}oisson.
\newblock {\em Nonlinear partial differential equations and their applications. {C}oll\`ege de {F}rance seminar, {V}ol. {VII} ({P}aris, 1983--1984)}, 122:1--3, 35--58, 1985.

\bibitem{bardos1985global}
C.~Bardos and P.~Degond.
\newblock Global existence for the {V}lasov-{P}oisson equation in 3 space variables with small initial data.
\newblock {\em Annales de l'Institut Henri Poincar{\'e} C, Analyse non lin{\'e}aire}, 2(2):101--118, 1985.

\bibitem{Bardos_Degond_Golse}
C.~Bardos, P.~Degond, and F.~Golse.
\newblock A priori estimates and existence results for the {V}lasov and {B}oltzmann equations.
\newblock {\em Nonlinear systems of partial differential equations in applied mathematics, {P}art 2 ({S}anta {F}e, {N}.{M}., 1984)}, 23:189--207, 1986.

\bibitem{batt1991global}
J.~Batt and G.~Rein.
\newblock Global classical solutions of the periodic {V}lasov-{P}oisson system.
\newblock {\em Comptes Rendus de l'Acad{\'e}mie des Sciences. S{\'e}rie I}, 313:411--416, 1991.

\bibitem{Bostan_2020}
M.~Bostan.
\newblock Asymptotic behavior for the {V}lasov-{P}oisson equations with strong uniform magnetic field and general initial conditions.
\newblock {\em Kinet. Relat. Models}, 13(3):531--548, 2020.

\bibitem{Bostan_Finot}
M.~Bostan and A.~Finot.
\newblock The effective {V}lasov-{P}oisson system for the finite {L}armor radius regime.
\newblock {\em Multiscale Model. Simul.}, 14(4):1238--1275, 2016.

\bibitem{Bostan_2016}
M.~Bostan, A.~Finot, and M.~Hauray.
\newblock Le syst\`eme de {V}lasov-{P}oisson effectif pour les plasmas fortement magn\'etis\'es.
\newblock {\em C. R. Math. Acad. Sci. Paris}, 354(8):771--777, 2016.

\bibitem{Bouchut_Golse_Pallard_2003}
F.~Bouchut, F.~Golse, and C.~Pallard.
\newblock Classical solutions and the {G}lassey-{S}trauss theorem for the 3{D} {V}lasov-{M}axwell system.
\newblock {\em Arch. Ration. Mech. Anal.}, 170(1):1--15, 2003.

\bibitem{plasma_2003}
T.~J.~M. Boyd and J.~J. Sanderson.
\newblock {\em The Physics of Plasmas}.
\newblock Cambridge University Press, 2003.

\bibitem{Caprino_2012}
S.~Caprino, G.~Cavallaro, and C.~Marchioro.
\newblock Time evolution of a {V}lasov-{P}oisson plasma with magnetic confinement.
\newblock {\em Kinet. Relat. Models}, 5(4):729--742, 2012.

\bibitem{castella1999propagation}
F.~Castella.
\newblock Propagation of space moments in the {V}lasov-{P}oisson equation and further results.
\newblock In {\em Annales de l'Institut Henri Poincare (C) Non Linear Analysis}, volume~16, pages 503--533. Elsevier, 1999.

\bibitem{Charles_co_numeric}
F.~Charles, B.~Despr\'es, R.~Dai, and S.~A. Hirstoaga.
\newblock Discrete moments models for {Vlasov} equations with non constant strong magnetic limit.
\newblock {\em Comptes Rendus. M\'ecanique}, 351(S1):307--329, 2023.

\bibitem{Charles_co_Landau}
F.~Charles, B.~Despr\'es, A.~Rege, and R.~Weder.
\newblock The magnetized {V}lasov-{A}mp\`ere system and the {B}ernstein-{L}andau paradox.
\newblock {\em J. Stat. Phys.}, 183(2):Paper No. 23, 57, 2021.

\bibitem{ChenChen2019}
Z.~Chen and J.~Chen.
\newblock Moments propagation for weak solutions of the {V}lasov-{P}oisson system in the three-dimensional torus.
\newblock {\em J. Math. Anal. Appl.}, 472(1):728--737, 2019.

\bibitem{Crippa_co}
G.~Crippa, M.~Inversi, C.~Saffirio, and G.~Stefani.
\newblock Existence and stability of weak solutions of the {V}lasov-{P}oisson system in localised {Y}udovich spaces.
\newblock {\em Nonlinearity}, 37(9):Paper No. 095015, 26, 2024.

\bibitem{Degond_Filbet_2016}
P.~Degond and F.~Filbet.
\newblock On the asymptotic limit of the three dimensional {V}lasov-{P}oisson system for large magnetic field: formal derivation.
\newblock {\em J. Stat. Phys.}, 165(4):765--784, 2016.

\bibitem{dobrushin1979vlasov}
R.~L. Dobrushin.
\newblock {V}lasov equations.
\newblock {\em Functional Analysis and Its Applications}, 13(2):115--123, 1979.

\bibitem{Fibet_Rodrigues_2016}
F.~Filbet and L.~M. Rodrigues.
\newblock Asymptotically stable particle-in-cell methods for the {V}lasov-{P}oisson system with a strong external magnetic field.
\newblock {\em SIAM J. Numer. Anal.}, 54(2):1120--1146, 2016.

\bibitem{Fibet_Rodrigues_2023}
F.~Filbet and L.~M. Rodrigues.
\newblock Asymptotically preserving particle methods for strongly magnetized plasmas in a torus.
\newblock {\em J. Comput. Phys.}, 480:Paper No. 112015, 23, 2023.

\bibitem{Filbet_co_2025}
F.~Filbet, L.~M. Rodrigues, and K.~H. Trinh.
\newblock A modified {C}rank-{N}icolson scheme for the {V}lasov-{P}oisson system with a strong external magnetic field, 2025.

\bibitem{Filbet_co_2021}
F.~Filbet, L.~M. Rodrigues, and H.~Zakerzadeh.
\newblock Convergence analysis of asymptotic preserving schemes for strongly magnetized plasmas.
\newblock {\em Numer. Math.}, 149(3):549--593, 2021.

\bibitem{Frenod_Sonnendrucker_98}
E.~Fr\'enod and E.~Sonnendr\"ucker.
\newblock Homogenization of the {V}lasov equation and of the {V}lasov-{P}oisson system with a strong external magnetic field.
\newblock {\em Asymptot. Anal.}, 18(3-4):193--213, 1998.

\bibitem{Frenod_Sonnendrucker_2001}
E.~Fr\'enod and E.~Sonnendr\"ucker.
\newblock The finite {L}armor radius approximation.
\newblock {\em SIAM J. Math. Anal.}, 32(6):1227--1247, 2001.

\bibitem{Galdi2011}
G.~P. Galdi.
\newblock {\em An Introduction to the Mathematical Theory of the Navier-Stokes Equations}.
\newblock Springer Monographs in Mathematics. Springer New York, 2011.

\bibitem{Glassey_Strauss}
R.~T. Glassey and W.~A. Strauss.
\newblock Singularity formation in a collisionless plasma could occur only at high velocities.
\newblock {\em Arch. Rational Mech. Anal.}, 92(1):59--90, 1986.

\bibitem{golse2016dynamics}
F.~Golse.
\newblock On the dynamics of large particle systems in the mean field limit.
\newblock {\em Macroscopic and large scale phenomena: coarse graining, mean field limits and ergodicity}, pages 1--144, 2016.

\bibitem{golse2017schrodinger}
F.~Golse and T.~Paul.
\newblock The {S}chr{\"o}dinger equation in the mean-field and semiclassical regime.
\newblock {\em Archive for Rational Mechanics and Analysis}, 223(1):57--94, 2017.

\bibitem{GSR98}
F.~Golse and L.~Saint-Raymond.
\newblock L'approximation centre-guide pour l'\'equation de {V}lasov-{P}oisson 2{D}.
\newblock {\em C. R. Acad. Sci. Paris S\'er. I Math.}, 327(10):865--870, 1998.

\bibitem{GSR99}
F.~Golse and L.~Saint-Raymond.
\newblock The {V}lasov-{P}oisson system with strong magnetic field.
\newblock {\em J. Math. Pures Appl. (9)}, 78(8):791--817, 1999.

\bibitem{GSR2003}
F.~Golse and L.~Saint-Raymond.
\newblock The {V}lasov-{P}oisson system with strong magnetic field in quasineutral regime.
\newblock {\em Math. Models Methods Appl. Sci.}, 13(5):661--714, 2003.

\bibitem{griffin2020recent}
M.~Griffin-Pickering and M.~Iacobelli.
\newblock Recent developments on the well-posedness theory for {V}lasov-type equations.
\newblock {\em From particle systems to partial differential equations}, 352:301--319, [2021] \copyright 2021.

\bibitem{Han_Kwan_2010}
D.~Han-Kwan.
\newblock The three-dimensional finite {L}armor radius approximation.
\newblock {\em Asymptot. Anal.}, 66(1):9--33, 2010.

\bibitem{HKD}
D.~Han-Kwan.
\newblock Quasineutral limit of the {V}lasov-{P}oisson system with massless electrons.
\newblock {\em Comm. Partial Differential Equations}, 36(8):1385--1425, 2011.

\bibitem{Han_Kwan_2012}
D.~Han-Kwan.
\newblock Effect of the polarization drift in a strongly magnetized plasma.
\newblock {\em ESAIM Math. Model. Numer. Anal.}, 46(4):929--947, 2012.

\bibitem{Han_Kwan_2013}
D.~Han-Kwan.
\newblock On the three-dimensional finite {L}armor radius approximation: the case of electrons in a fixed background of ions.
\newblock {\em Ann. Inst. H. Poincar\'e{} C Anal. Non Lin\'eaire}, 30(6):1127--1157, 2013.

\bibitem{HKD_HDR}
D.~Han-Kwan.
\newblock Stabilit\'e, limites singuli\`eres et conditions de contr\^ole g\'eom\'etrique en th\'eorie cin\'etique, 09 2017.
\newblock M\'emoire pr\'esente\'e \`a Universit\'e Paris-Diderot pour l’obtention de l’habilitation \`a diriger des recherches.

\bibitem{han2017quasineutral}
D.~Han-Kwan and M.~Iacobelli.
\newblock Quasineutral limit for {V}lasov-{P}oisson via {W}asserstein stability estimates in higher dimension.
\newblock {\em Journal of Differential Equations}, 263(1):1--25, 2017.

\bibitem{Herda_2016}
M.~Herda.
\newblock On massless electron limit for a multispecies kinetic system with external magnetic field.
\newblock {\em J. Differential Equations}, 260(11):7861--7891, 2016.

\bibitem{Holding_Miot}
T.~Holding and E.~Miot.
\newblock Uniqueness and stability for the {V}lasov-{P}oisson system with spatial density in {O}rlicz spaces.
\newblock {\em Mathematical analysis in fluid mechanics---selected recent results}, 710:145--162, [2018] \copyright 2018.

\bibitem{Horst_Hunze}
E.~Horst and R.~Hunze.
\newblock Weak solutions of the initial value problem for the unmodified nonlinear {V}lasov equation.
\newblock {\em Math. Methods Appl. Sci.}, 6(2):262--279, 1984.

\bibitem{iacobelli_new_2022}
M.~Iacobelli.
\newblock A new perspective on {Wasserstein} distances for kinetic problems.
\newblock {\em Arch Rational Mech Anal}, 244(1):27--50, Apr. 2022.

\bibitem{iacobelli_stability_2024}
M.~Iacobelli and J.~Junné.
\newblock Stability estimates for the {Vlasov}-{Poisson} system in p‐kinetic {Wasserstein} distances.
\newblock {\em Bulletin of London Math Soc}, 56(7):2250--2267, July 2024.

\bibitem{iacobelli2024enhanced}
M.~Iacobelli and L.~Lafleche.
\newblock Enhanced stability in quantum optimal transport pseudometrics: From {Hartree} to {V}lasov-{P}oisson.
\newblock {\em Journal of Statistical Physics}, 191(12):157, 2024.

\bibitem{iordanskii1961cauchy}
S.~V. Iordanskii.
\newblock The {C}auchy problem for the kinetic equation of plasma.
\newblock {\em Trudy Matematicheskogo Instituta Imeni VA Steklova}, 60:181--194, 1961.

\bibitem{jabin2014review}
P.-E. Jabin.
\newblock A review of the mean field limits for {V}lasov equations.
\newblock {\em Kinetic and Related models}, 7(4):661--711, 2014.

\bibitem{Galactic_2008}
B.~James and T.~Scott.
\newblock {\em Galactic Dynamics: Second Edition}.
\newblock Princeton University Press, 2008.

\bibitem{junne_stability_2025}
J.~Junné and A.~Rege.
\newblock Stability estimates in kinetic {Wasserstein} distances for the {Vlasov}-{Poisson} system with {Yudovich} density, Mar. 2025.
\newblock arXiv:2503.07122 [math].

\bibitem{KS02}
S.~Klainerman and G.~Staffilani.
\newblock A new approach to study the {V}lasov-{M}axwell system.
\newblock {\em Commun. Pure Appl. Anal.}, 1(1):103--125, 2002.

\bibitem{Knopf_Weber_2022}
P.~Knopf and J.~Weber.
\newblock On the two and one-half dimensional {V}lasov-{P}oisson system with an external magnetic field: global well-posedness and stability of confined steady states.
\newblock {\em Nonlinear Anal. Real World Appl.}, 65:Paper No. 103460, 24, 2022.

\bibitem{lafleche2019propagation}
L.~Lafleche.
\newblock Propagation of moments and semiclassical limit from {H}artree to {V}lasov equation.
\newblock {\em Journal of Statistical Physics}, 177(1):20--60, 2019.

\bibitem{lafleche2021global}
L.~Lafleche.
\newblock Global semiclassical limit from {Hartree} to {V}lasov equation for concentrated initial data.
\newblock In {\em Annales de l'Institut Henri Poincar{\'e} C, Analyse non lin{\'e}aire}, volume~38, pages 1739--1762. Elsevier, 2021.

\bibitem{lions1991propagation}
P.-L. Lions and B.~Perthame.
\newblock Propagation of moments and regularity for the 3-dimensional {V}lasov-{P}oisson system.
\newblock {\em Inventiones mathematicae}, 105(1):415--430, 1991.

\bibitem{loeper2006uniqueness}
G.~Loeper.
\newblock Uniqueness of the solution to the {V}lasov-{P}oisson system with bounded density.
\newblock {\em Journal de math{\'e}matiques pures et appliqu{\'e}es}, 86(1):68--79, 2006.

\bibitem{LS14}
J.~Luk and R.~M. Strain.
\newblock A new continuation criterion for the relativistic {V}lasov-{M}axwell system.
\newblock {\em Comm. Math. Phys.}, 331(3):1005--1027, 2014.

\bibitem{LS16}
J.~Luk and R.~M. Strain.
\newblock Strichartz estimates and moment bounds for the relativistic {V}lasov-{M}axwell system.
\newblock {\em Arch. Ration. Mech. Anal.}, 219(1):445--552, 2016.

\bibitem{Miot_2016}
E.~Miot.
\newblock A uniqueness criterion for unbounded solutions to the {V}lasov-{P}oisson system.
\newblock {\em Comm. Math. Phys.}, 346(2):469--482, 2016.

\bibitem{pallard2012moment}
C.~Pallard.
\newblock Moment propagation for weak solutions to the {V}lasov-{P}oisson system.
\newblock {\em Communications in Partial Differential Equations}, 37(7):1273--1285, 2012.

\bibitem{pfaffelmoser1992global}
K.~Pfaffelmoser.
\newblock Global classical solutions of the {V}lasov-{P}oisson system in three dimensions for general initial data.
\newblock {\em Journal of Differential Equations}, 95(2):281--303, 1992.

\bibitem{Preissl_co_2021}
D.~Preissl, C.~Cheverry, and S.~Ibrahim.
\newblock Uniform lifetime for classical solutions to the hot, magnetized, relativistic {V}lasov-{M}axwell system.
\newblock {\em Kinet. Relat. Models}, 14(6):1035--1079, 2021.

\bibitem{rege2021uniformB}
A.~Rege.
\newblock The {V}lasov-{P}oisson system with a uniform magnetic field: propagation of moments and regularity.
\newblock {\em SIAM J. Math. Anal.}, 53(2):2452--2475, 2021.

\bibitem{rege2023propagation}
A.~Rege.
\newblock Propagation of velocity moments and uniqueness for the magnetized {V}lasov-{P}oisson system.
\newblock {\em Communications in Partial Differential Equations}, 48(3):386--414, 2023.

\bibitem{rege2025stability}
A.~Rege.
\newblock Stability estimates for magnetized {V}lasov equations.
\newblock {\em Journal of Differential Equations}, 425:763--788, 2025.

\bibitem{robinson2016three}
J.~C. Robinson, J.~L. Rodrigo, and W.~Sadowski.
\newblock {\em The three-dimensional {N}avier-{S}tokes equations: {C}lassical theory}, volume 157.
\newblock Cambridge university press, 2016.

\bibitem{Rudolf}
K.~Rudolf.
\newblock Das {A}nfangswertproblem der {S}tellardynamik.
\newblock {\em Z. Astrophys.}, 30:213--229, 1952.

\bibitem{Ryutov_1999}
D.~D. Ryutov.
\newblock Landau damping: half a century with the great discovery.
\newblock {\em Plasma Physics and Controlled Fusion}, 41(3A):A1--A12, jan 1999.

\bibitem{Schaffer}
J.~Schaeffer.
\newblock Global existence of smooth solutions to the {V}lasov-{P}oisson system in three dimensions.
\newblock {\em Comm. Partial Differential Equations}, 16(8-9):1313--1335, 1991.

\bibitem{ukai1978classical}
S.~Ukai and T.~Okabe.
\newblock On classical solutions in the large in time of two-dimensional {V}lasov's equation.
\newblock {\em Osaka Math. J.}, 15(2):245--261, 1978.

\bibitem{Villani_notes}
C.~Villani.
\newblock Landau damping, {N}otes de cours.
\newblock CEMRACS, 2010.
\newblock \url{http://www.cedricvillani.org/sites/dev/files/old_images/2012/08/B13.Landau.pdf}.

\end{thebibliography}

\end{document}